\documentclass{imsart}
\usepackage{amsmath}
\usepackage{amsthm}
\usepackage{amsfonts}
\usepackage{amscd}
\usepackage{amssymb}
\usepackage{amstext}
\usepackage{hyperref}
\usepackage{epsfig}

\numberwithin{equation}{section}

\allowdisplaybreaks

\usepackage{marginnote}
\usepackage{comment}

\newcommand{\R}{\mathbb{R}}
\newcommand{\E}{\mathbb{E}}
\newcommand{\Pb}{\mathbb{P}}
\newcommand{\Z}{\mathbb{Z}}
\newcommand{\N}{\mathbb{N}}

\newcommand{\C}{\mathcal{C}}
\newcommand{\B}{\mathcal{B}}
\newcommand{\I}{\mathcal{I}}
\newcommand{\Lv}{\mathcal{L}}
\newcommand{\T}{\mathcal{T}}
\newcommand{\X}{\mathcal{X}}
\newcommand{\W}{\mathcal{W}}
\newcommand{\K}{\mathcal{K}}
\newcommand{\be}{\begin{equation}}
\newcommand{\ee}{\end{equation}}

\newtheorem{theorem}{Theorem}[section]

\newtheorem{lemma}[theorem]{Lemma}
\newtheorem{proposition}[theorem]{Proposition}
\newtheorem{corollary}[theorem]{Corollary}
\newtheorem*{assumption(pp)}{Assumption HT}
\newtheorem*{assumption(d)}{Assumption L}
\newtheorem*{assumption(pp*)}{Assumption $\widetilde{\text{HT}}$}
\newtheorem*{assumption(d*)}{Assumption $\tilde{\text{L}}$}

\theoremstyle{definition}
\newtheorem{definition}[theorem]{Definition}

\theoremstyle{remark}
\newtheorem{remark}[theorem]{Remark}

\numberwithin{equation}{section}

\def\deg{\mathop{\mathrm{deg}}\nolimits}

\def\IIC{{\textrm {\tiny IIC}}}
\def\IPC{{\textrm {\tiny IPC}}}

\def\<{\langle}
\def\>{\rangle}

\begin{document}

\title{Backbone scaling limits for random walks on random critical trees}

\author{G\'erard Ben Arous \thanksref{t1}}
\thankstext{t1}{G.B.A.~is supported NSF DMS1209165, BSF 2014019}
\address{G. Ben Arous\\
  Courant Institute of Mathematical Sciences\\
  New York University\\
  251 Mercer Street\\
  New York, NY 10012, USA}

\author{Manuel Cabezas\thanksref{t2}}
\thankstext{t2}{M.C.~is supported by Fondecyt fellowship \#1201090.} 
\address{Manuel Cabezas\\
  Pontificia Universidad Cat\'olica de Chile\\
  Avenida Vicu\~na Mackenna 4860\\
  Santiago, Chile}

\author{Alexander Fribergh\thanksref{t3}}
\thankstext{t3}{A.F.~is supported by NSERC discovery grant and FRQNT jeune chercheur.}
\address{Alexander Fribergh\\Universit\'e de Montr\'eal, DMS\\
Pavillon Andr\'e-Aisenstadt\\     2920, chemin de la Tour Montr\'eal (Qu\'ebec),  H3T 1J4}

\pagestyle{plain}
\begin{abstract}
We prove the existence of scaling limits for the projection on the backbone of the random walks on the Incipient Infinite Cluster and the Invasion Percolation Cluster on a regular tree. We treat these projected random walks as Randomly trapped random walks (as defined in~\cite{rtrw}) and thus describe these scaling limits as spatially subordinated Brownian motions.
\end{abstract}

\maketitle
\section{Introduction}

We study here the sub-diffusive behavior of standard random walks on random critical trees. 
More precisely, we aim to derive and describe scaling limits for these random walks, and relate them to the class of processes called Spatially Subordinated Brownian Motions (SSBM), and the limit theorems introduced in~\cite{rtrw}.
We consider this question on two classes of random subtrees of $\mathbb{T}_2$, the rooted infinite binary tree-graph, namely the Incipient Infinite Cluster (the IIC) and the Invasion Percolation Cluster (the IPC).
Firstly, we consider the case of random walks on the critical percolation cluster on $\mathbb{T}_2$. Following the beautiful early work by Kesten(\cite{Kesten}), we consider the simple random walk on the critical percolation cluster of the root, conditioned to be infinite (the infinite incipient cluster). This random walk is subdiffusive and Barlow and Kumagai (\cite{BarlowKumagai2006}) have established sharp sub-diffusive heat kernel estimates in this context. The IIC~is an infinite tree with a unique simple path to infinity, the backbone. We prove here a scaling limit for the projection of the random walk  on the backbone. We see this projection as a Randomly Trapped Random Walk. This allows us to use the results of~\cite{rtrw}, describe precisely this scaling limit and show that it belongs to the class of SSBMs.
In fact, there is an alternate way to study the random walk on critical percolation clusters, and to understand its sub-diffusivity. We can condition the finite cluster of the root to be of size $N$, and let $N$ tend to infinity. This random tree, properly rescaled, converges to the Continuum Random Tree (CRT) introduced by Aldous (\cite{AldousCRT1}). Furthermore, D. Croydon (\cite{rwrt}) proved that the random walk, properly rescaled, converges to the Brownian motion on the CRT (introduced by Krebs (\cite{krebs1995brownian})). 
We relate these two approaches. First we show how the SSBM scaling limit can be obtained as the projection of the Brownian Motion on the Continuum Random Forest  to its backbone. We then introduce a richer class of processes, the SSBMs on trees, and show that, if one picks $K$ points at random  on the percolation cluster of the root conditioned to be large, and project the random walk on the geodesic tree defined by these $K$ points, the scaling limit exists and belongs to the class of SSBMs on the geodesic tree defined by $K$ points picked at random on the CRT.

Secondly, we can also treat the case of random walks on the invasion percolation cluster (IPC) on $\mathbb{T}_2$. This is a well known case of self organized criticality, see for instance the recent works (\cite{ipc} and \cite{AngelGoodmanMerle2013}) which give a scaling limit for the IPC~itself. We show that the random walk projected on the backbone and properly normalized converges to a slight variant of an SSBM.

\section{Statement of Results}

\subsection{Results for the Incipient Infinite Cluster}
Let $\mathbb{T}_2$ be a rooted binary tree, i.e., $\mathbb{T}_2$ is an infinite tree in which every vertex has degree $3$, except for the root $\rho$ which has degree $2$. Denote by $\C_\rho$ the connected component of the root $\rho$ under critical percolation in $\mathbb{T}_2$. Let also $\C_{\rho}^n$ be a random tree having the law of $\C_{\rho}$ conditioned on intersecting the boundary of a ball of radius $n$ (centered at the root, with the graph-distance on $\mathbb{T}_2$).
The \emph{Incipient Infinite Cluster} (IIC) (which we will denote by $\C^\infty$) is an infinite random tree which is obtained as the limit as $n\to\infty$ of $\C_{\rho}^n$. For details of the definition we refer to \cite{Kesten}. We will denote $(\mathcal{X},\mathcal{G},\boldmath{P})$ the probability space in which $\C_{\rho}^\infty$ is defined.

It is a known fact that the IIC possesses a single path to infinity, i.e., there exists a unique nearest-neighbor, non-self intersecting path starting at the root which is unbounded. This path is called the \textit{backbone}. Obviously, the backbone is isomorphic, as a graph, to $\N$. Hence, the IIC can be seen as $\N$ adorned with finite branches. The $k$-th branch (that is, the branch emerging from the $k$-th vertex of the backbone) will be denoted $\B_k$ and the $k$-th vertex of the backbone will be regarded as the root of $\B_k$. In \cite{Kesten} it is showed that $(\B_k)_{k\in\N}$ is distributed as an i.i.d.~sequence of critical percolation clusters on $\mathbb{T}_2^\ast$, where $\mathbb{T}_2^\ast$ is an infinite rooted tree in which each vertex has degree $3$ except for the root which has degree $1$.

We will study the random walk on the IIC projected to the backbone.
Let $(Y^{\IIC}_k)_{k\in\N_0}$ be a discrete time, nearest neighbor, symmetric random walk on $\C^\infty$ starting at the root. Denote $\Phi:\C^\infty\mapsto\N$ the projection of $\C^\infty$ onto the backbone. That is, $\Phi(x)=k$ iff $x\in\B_k$.
Finally, let  $(X^{\IIC}_t)_{t\geq0}$ be the simple random walk on the IIC projected to the backbone, i.e., $X^{\IIC}_t:=\Phi(Y^{\IIC}_{\lfloor t \rfloor})$. For each $\omega\in \mathcal{X}$ let $\boldmath{P}_\omega$ denote the law of $X^{\IIC}$ for a fixed realization $\C_{\rho}^{\infty}(\omega)$ of the IIC. We define the annealed law of $X^{\IIC}$ as the semi-direct product $\Pb:=\boldmath{P}\times\boldmath{P}_\omega$. More precisely,
\[\Pb[A]=\int_{\mathcal{X}} P_w(A) P(dw),\]
for all $A$ Borelian of $D(\R_+)$ endowed with the uniform convergence, where $D(\R_+)$ denotes the space of c\`{a}dl\`{a}g paths.
The first main theorem of this article is the following:
\begin{theorem}\label{thm:IIC}
 There exists a random process $(Z^{\IIC}_t)_{t\geq0}$ such that \[(\epsilon^{1/3} X^{\IIC}_{\lfloor \epsilon^{-1} t\rfloor })_{t\geq0}{\to} (Z^{\IIC}_t)_{t\geq0} \quad \text{as }\epsilon \to 0,\]
 in $\Pb$-distribution, and the convergence takes place in the space $D(\R_+)$ endowed with the topology of uniform convergence.
\end{theorem}

\begin{remark} We believe that this result should hold for critical Galton-Watson trees under some moment condition (finite, positive variance). We restrict ourselves to the case of a binary tree for simplicity and readability. \end{remark} 

 We remark here that the process $Z^{\IIC}$ belongs to a class of processes introduced in \cite{rtrw} called \emph{Spatially Subordinated Brownian motions} (SSBM) and which are obtained as highly non-trivial time changes of a Brownian motion.
  Now we prepare the ground for a detailed description of $Z^{\IIC}$.
Let $(\bar{x}_i,\bar{y}_i)_{i\in\N}$ be an inhomogeneous Poisson point process on $\R\times \R_+$ with intensity measure $\frac{1}{2}\pi^{-1/2}y^{-3/2}dxdy$. Define a random measure $\mu_{\IIC}$ as
\begin{align}\label{eq:varrho}
\mu_{\IIC}=\sum_{i\in\N} \bar{y}_i\delta_{\bar{x}_i}.
\end{align}
Let $\left((S^i(t))_{t\geq0}\right)_{i\in\N}$ be an i.i.d.~sequence of random processes having the annealed law of the inverse local time at the root of the Brownian motion on the \emph{Continuum Random Tree} (CRT) (see display \eqref{annealedlocaltime} in Section \ref{s:BMCRT} for the definition of that process). We also assume that the $\left((S^i(t))_{t\geq0}\right)_{i\in\N}$ are independent of $\mu_{\IIC}$.

Finally, let $(B^+_t)_{t\geq0}$ be a one-dimensional, standard Brownian motion reflected at the origin independent of everything else and let $l^+(x,t)$ be its local time.
Define
\begin{equation}\label{eq:defphiiic}
\phi^{\IIC}_t:=\sum_{i\in\N}\bar{y}_i^{3/2}S^i(\bar{y}_i^{\scriptscriptstyle-\frac{1}{2}}l^+(\bar{x}_i,t))
\end{equation}
and its right-continuous generalized inverse
\[
\psi^{\IIC}_t:=\inf\{s\geq0:\phi^{\IIC}_s > t\}.
\]
The following theorem provides a description of $Z^{\IIC}$.
\begin{theorem}\label{t:descriptionofziic}
The process $Z^{\IIC}$ can be defined as the SSBM
\[
Z^{\IIC}_t:=B^+_{\psi^{\IIC}_{t}}, \quad t\geq0.
\]
\end{theorem}

This theorem is directly obtained when proving Theorem~\ref{thm:IIC}.

Note that $Z^{\IIC}$ is a time-change of $B^+$ in which each site $\bar{x}_i$ plays the role of a \textit{trap} where $Z^{\IIC}$ will spend a positive amount of time. The time spent on $\bar{x}_i$ will depend on $S^i$ (which is the inverse local time at the root of a Brownian motion on the CRT), on $\bar{y}_i$ (which, as we will see later, can be thought of as being the \textit{depth of the trap} at $\bar{x}_i$) and on $l(\bar{x}_i,t)$ (which, in some sense, measures the ``number of times that $Z^{\IIC}$ has passed through the trap at $\bar{x}_i$'').

We provide yet another, alternative, representation of $Z^{\IIC}$ as the \emph{Brownian motion in the Continuum Random Forest projected to the backbone}. The Continuum Random Forest can be informally described as a collection of Continuum Random Trees glued to $\R_+$ and can be constructed as follows: Let $(\bar{x}_i,\bar{y}_i)_{i\in\N}$ be as in \eqref{eq:varrho} and $((\T_i,d_i,\mu_i))_{i\in\N}$ be an i.i.d.~sequence of Continuum Random Trees independent of the $(\bar{x}_i,\bar{y}_i)_{i\in\N}$ (for the definition of the CRT see Definition \ref{definitioncrt} in Section \ref{s:crt}). Then we glue the root of the rescaled random trees $(\T_i, \bar{y}_i^{\scriptscriptstyle\frac{1}{2}}d_i,\bar{y}_i\mu_i)$ to the backbone $\R_+$ at positions $\bar{x}_i$. The tree $\mathcal{F}$ obtained in this way is the Continuum Random Forest. We endow $\mathcal{F}$ with a distance $d$ which is obtained from the Euclidean distance in $\R_+$ and the $\bar{y}_i^{\scriptscriptstyle\frac{1}{2}}d_i,i\in\N$ in the obvious way. We also endow $\mathcal{F}$ with a measure $\mu$ which is the sum of the $\bar{y}_i\mu_i,i\in\N$. Note that we are assigning $\mu$-measure $0$ to the backbone $\R_+$.

   In \cite{AldousCRT2}, Aldous defined the \emph{Brownian motion} on locally compact continuous trees as a strong Markov process satisfying a set of properties (We recall his definition in Section \ref{s:BMCRT}, Definition \ref{definitionofbmonadendrite}) and he also claimed that one can prove such process is unique. Existence was first provided by Krebs in \cite{krebs1995brownian} (see also \cite{rwrt} and \cite{kigami1995harmonic} for a different approach). This allow us to consider $(B^{\mathcal{F}}_t)_{t\geq0}$ the \emph{Brownian motion in the Continuum Random Forest}. Let $\pi: \mathcal F \to \R_+$ be the projection to the backbone.
 \begin{theorem}\label{prop:alternativeexrepssionforziic}
   The process $(Z^{\IIC}_t)_{t\geq0}$ is distributed as $(\pi(B_t^{\mathcal{F}}))_{t\geq0}$.
\end{theorem}
\begin{remark}
We would like to point out that Theorem \ref{prop:alternativeexrepssionforziic} suggests a different way of proving Theorem \ref{thm:IIC}. Suppose we can show that the random walk on the IIC scales to the Brownian motion in the Continuum random forest. To our knowledge, the latter has never been rigorously proved, but Theorem 7.2 in \cite{scaling}, or the arguments in \S 7.4 of \cite{athreya2017invariance} could be used to produce a proof. Assume in addition that, in the space where the convergence takes place, the discrete backbone coincides with the backbone of the IIC and that the projection to the backbone is a continuous function. Then, Theorem \ref{thm:IIC} would follow from the convergence of the RW on the IIC to the BM on the CRF by the continuous mapping theorem.
\end{remark}

\subsection{Results for the Invasion Percolation Cluster}
Now, we turn our attention to the \textit{Invasion Percolation Cluster} (IPC). The IPC was introduced in \cite{WilkinsonWillemsen1983} and is obtained through an \textit{invasion process} in the vertices of $\mathbb{T}_2$.
Let $(w_x)_{x\in \mathbb{T}_2}$ be an i.i.d.~sequence of random variables indexed by the vertices of $\mathbb{T}_2$ whose common distribution is uniform on $(0,1)$. Set $\I^0:=\rho$ and
\[\I^{n+1}:=\I^n\cup\left\{x:d(x,\I^n)=1\textrm{ and } w_x=\min\{w_z:d(\I ^n,z)=1\}\right\}\]
where $d$ is the graph distance in $\mathbb{T}_2$.
That is, $\I^{n+1}$ is obtained from $\I^n$ by adding the vertex $x$ on the outer boundary of $\I^n$ with smaller ``weight'' $w_x$.
The \textit{ Invasion Percolation Cluster} (IPC) on $\mathbb{T}_2$ is defined as $\I^{\infty}:=\cup_{n\in\N}\I^n$. We denote by $(\mathcal{X},\mathcal{G},\boldmath{P})$ the probability space in which $\I^{\infty}$ is defined.

It was shown in \cite[\S1.3]{ipc} that, similarly to the IIC, the IPC has a single path to infinity. Thus, as the IIC, the IPC can also be seen as a backbone $\N$ adorned with finite branches, but in this case the branches are not i.i.d. In fact, denoting $\Lv_k$ the branch which emerges from the $k$-th vertex of the backbone, we have that $\Lv_k$ is distributed as a sub-critical percolation cluster with a percolation parameter which depends on $k$ and tends to the critical value as $k\to\infty$ (see \cite{ipc} for a proof of that fact).
 A detailed description of the scaling limit through its contour function is given in \cite{AngelGoodmanMerle2013}

Next, we define the projection on the backbone of the simple random walk on the IPC.
Let $\Phi:\I^\infty\mapsto\N$ be the projection to the backbone on $\I^{\infty}$.
Let $(Y^{\IPC}_k)_{k\in\N}$ be a discrete-time, nearest neighbor, symmetric random walk on $\I^\infty$ starting at the root. Let $(X^{\IPC}_t)_{t\geq 0}$ be defined by setting $X^{\IPC}_t:=\Phi(Y^{\IPC}_{\lfloor t\rfloor})$. For each $\omega\in \mathcal{X}$ let $\boldmath{P}_w$ denote the law of $X^{\IPC}$ for a fixed realization $\I^{\infty}(\omega)$ of the IPC. We define the annealed law of $X^{\IPC}$ as the semi-direct product $\Pb:=\boldmath{P}\times\boldmath{P}_\omega$.

\begin{theorem}\label{thm:IPC}
There exists a random process $(Z^{\IPC}_t)_{t\geq0}$ such that \[(\epsilon^{1/3} X^{\IPC}_{\lfloor \epsilon^{-1} t\rfloor })_{t\geq0}\to (Z^{\IPC}_t)_{t\geq0} \quad \text{as }\epsilon \to 0,\]
in $\Pb$-distribution, and the convergence takes place in the space $D(\R_+)$ of c\`{a}dl\`{a}g paths endowed with the topology of uniform convergence.
\end{theorem}

Contrary to $Z^\IIC$, the process $Z^{\IPC}$ is not an SSBM in the strict sense of~\cite{rtrw}. However, the construction of the process is very similar and we will refer to this sort of process as an SSBM as well. The only difference is that the measure $\mu_{\IIC}$ (see display \eqref{eq:varrho}) used to choose $Z^\IIC$ will be replaced by a slightly more complex random measure $\mu_{\IPC}$ which is neither independent nor i.i.d. 

Let $E_t$ be the lower envelope of an homogeneous Poisson point process in $(0,\infty)\times(0,\infty)$. More specifically,
let $\mathcal{P}$ be a Poisson point process on $(0,\infty)\times(0,\infty)$ with intensity 1 and take $t>0$, we define
\begin{equation} \label{lt}
E_t:=\min\{y : (x,y)\in {\mathcal P} \textrm{ for some } x\leq t\}.
\end{equation}
Note that the process $E$ is decreasing and piecewise constant.

 Let $\delta>0$, $\gamma\geq0$ and $(I^{\delta,\gamma}_t)_{t\geq0}$ be the subordinator characterized by
\[\E[\exp(-\lambda I^{\delta,\gamma}_t)]=\exp(-t\delta(\sqrt{2\lambda+\gamma^2}-\gamma))\quad \text{for all }\lambda>0.\]
The process $(I^{\delta,\gamma}_t)_{t\geq0}$ is called inverse Gaussian subordinator of parameters $\delta,\gamma$. For more details we refer to \cite{levy}, example 1.3.21.

 For each realization of $E$, let $(b_i)_{i\in\N}$ be an enumeration of the points of discontinuity of $E$ and $a_i:=\max\{b_j: b_j<b_i\}$ so that $E$ is constant on the intervals $[a_i,b_i)$, $i\in\N$.
Also let $((I^i_t)_{t\geq0})_{i\in\N}$ be an independent family of inverse Gaussian subordinators, each one with parameters $\delta=1/\sqrt{2}$ and $\gamma=\sqrt{2}E_{a_i}$.
Let $\mu_{\IPC}^i$ be the random Lebesgue-Stieltjes measure associated to $I^i$.
Finally, we define\footnote{Alternatively, $\mu_{\IPC}$ can be described as a Cox process directed by the random measure $\xi=\sum_i1_{[a_i,b_i)}(x) dx \zeta^i(dy) $, where $\zeta^i(dy)$ is the L\'{e}vy measure of $I^i$.} $\mu_{\IPC}(A):=\sum_{i\in\N}\mu_{\IPC}^i(A\cap [a_i,b_i))$ for each Borelian set $A$.
Since the inverse Gaussian subordinators are pure jump processes, we have that $\mu_{\IPC}$ is a purely atomic measure. Hence, we can write
\be\label{ipclimitmeasure}
\mu_{\IPC}=\sum_{i\in\N}\tilde{y}_i\delta_{\tilde{x}_i}.
\ee

Let $((S^i(t))_{t\geq0})_{i\in\N}$ be an i.i.d.~sequence of random processes having the (annealed) law of the inverse  local time at the root of the Brownian motion on the CRT and independent of $\mu_{\IPC}$. Let $(B^+_t)_{t\geq0}$ be a one-dimensional standard Brownian motion reflected at the origin independent of everything else and let $l^+(x,t)$ be its local time.
Define
\[
\phi^{\IPC}_t:=\sum_{i\in\N}\tilde{y}_i^{ 3/2}S^i(\tilde{y}_i^{\scriptscriptstyle-1/2}l^+(\tilde{x}_i,t))
\]
and its right-continuous generalized inverse
\[
\psi^{\IPC}_t:=\inf\{s\geq0:\phi^{\IPC}_s >t\}.
\]

 
 The next theorem provides the description of $Z^{\IPC}$.
\begin{theorem}\label{t:descriptionofzipc}
The process $Z^{\IPC}$ can be defined as the SSBM
\[
Z^{\IPC}_t:=B^+_{\psi^{\IPC}_{t}}, \quad t\geq0.
\]
\end{theorem}
The theorem above will be obtained together with Theorem \ref{thm:IPC}, therefore, for its proof we refer to the proof of Theorem \ref{thm:IPC}.

 We also get a representation of $Z^{\IPC}$ as the Brownian motion in a Random Forest projected to
  the backbone. Let $(\tilde{\mathcal{F}},\tilde{d},\tilde{\mu})$ be the tree constructed exactly as the Continuum Random Forest $\mathcal{F}$, with the only difference that instead of choosing the locations and sizes of the trees according to $(\bar{x}_i,\bar{y}_i)_{i\in\N}$ as in \eqref{eq:varrho}, we use $(\tilde{x}_i,\tilde{y_i})_{i\in\N}$ as in \eqref{ipclimitmeasure}. Let $(B^{\tilde{\mathcal{F}}}_t)_{t\geq0}$ be the Brownian motion in $(\tilde{\mathcal{F}},\tilde{d},\tilde{\mu})$ and $\pi:\tilde{\mathcal{F}}\to\R_+$ be the projection to the backbone.
 \begin{theorem}\label{prop:alternativeexrepssionforzipc}
 The process $(\pi(B^{\tilde{\mathcal{F}}}_t))_{t\geq0}$ is distributed as $(Z^{\IPC}_t)_{t\geq0}$.
 \end{theorem}
 
 This theorem is obtained as a by-product of the proof of Theorem~\ref{thm:IPC}.

\subsection{Scaling limits on large random trees} \label{Sect_finite_SSBM}

We study in Section 8 a problem closely related to the scaling limit questions discussed above. Instead of studying infinite trees and projecting the random walk on the backbone, we consider random walks on finite random trees conditioned to be large. We show how the notion of SSBM can be usefully extended to this context. It has been shown by Croydon~\cite{rwrt} that the random walk on a critical Galton-Watson tree conditioned to be large converges, once properly normalized, to the Brownian motion on the CRT.
In this context, the notion of backbone is not as immediate as in the case of the IIC. A simple substitute is to pick one point at random in the critical discrete tree and look at the projection of the random on the geodesic linking this point to the origin, i.e.~the ancestry line of this point. As we will see, the scaling limit of this projection requires a straightforward generalization of the notion of SSBM.

We can then extend this construction in an interesting way. Pick now $K$ points at random in the large finite critical tree, and consider the geodesic tree defined by the root and these $K$ points, i.e.~the genealogical tree. We show that the projection of the random walk on this geodesic tree with $K$ leaves converges to an interesting generalization of the notion of SSBM, which we call a SSBM on a finite tree.

The convergence for all $K$ to SSBMs on trees contains roughly the same information as the convergence to the Brownian motion on the CRT up to tightness considerations.  More precisely, if a discrete sequence of processes on the tree is such that, for any $K$, its projection onto the $K$-skeleton converges to the $K$-SSBM, then, the only possible scaling limit for the sequence of processes is the BM on the CRT.
 This new notion opens up the possibility of proving scaling limits along the line opened in~\cite{highdimensionallabyrinth} but for models more difficult than the one considered in~\cite{simplelabyrinth}. 
 
 The discussion of SSMB on trees in Section \ref{sect_finite_SSBM} will have an informal tone. We will not prove any mayor theorem, and instead we will discuss some conjectures.



We begin here to introduce the notion of SSBM on a finite tree. Assume, that we are given a metric tree $\mathrm{T}_K$, which is composed of a finite number of edges $e_1,\ldots, e_K$, all of which have a given length $l_1,\ldots,l_K$ for some $K<\infty$. We can obtain a metric on $\mathrm{T}_K$ by defining the distance linearly along every edge. As in the case of the standard construction of SSBM, the ingredients of the construction are a probability measure $\mathbb{F}$ on the set of Laplace exponents of subordinators, a point process on $\mathrm{T}_K$ and a constant $\gamma\in(0,1)$. 

 We can then repeat the procedure of the previous section: generate a collection of points $(x_i,y_i)_{i\in \mathbb{N}}$ arising from a $\gamma$-stable Levy process on $\mathrm{T}_K$ conditioned on having total volume 1, (i.e., $\sum_{i\in\N}y_i=1$) this object\footnote{Although the event under which we are conditioning has probability $0$, we can still make sense of it. See Section \ref{sect_finite_SSBM}.} has a law which is denoted $\mathbb{M}^{(\gamma)}$. Once the points $(x_i,y_i)_{i\in\N}$ are sampled, consider independent subordinators $((S^i(t))_{t\geq 0})_{i\in \N}$ with independent Laplace exponents $(f_i)_{i\in \N}$ sampled according to $\mathbb{F}$. Finally, let $(B^{\mathrm{T}_K}_t)_{t\geq 0}$ be a Brownian motion on $\mathrm{T}_K$ (with equiprobable transition probabilities at intersections) independent of everything else and $l(x,t)$ its local time. We then set
 \[
 \phi_t:=\sum_{i\in \N} y_i^{1+\gamma}S^i(y_i^{-\gamma}l(x_i,t)) \text{ and } \psi_t:=\inf\{s\geq 0,\ \phi_s>t\}.
 \]
 The time changes $\phi$ and $\psi$ depend on $\mathrm{T}_K,\mathbb{F},\mathbb{M}^{(\gamma)}$, although not reflected in the notation.
 \begin{definition}
 The process $(B_{\psi_t}^{\mathrm{T}_K})_{t\geq 0}$ is called an SSBM on the tree $\mathrm{T}_K$ (or $\mathrm{T}_K$-SSBM).
\end{definition}

SSBMs on trees are natural counterparts of the SSBMs appearing in Theorem~\ref{thm:IIC} and Theorem~\ref{thm:IPC}.

The SSBMs appear naturally when considering the projection of a random walk on a random tree onto a subtree. Let us consider $\T_n$ a random tree conditioned to be of volume $n$. Let us assume that $n^{-1/2}\T_n$ converges to the Continuum Random Tree. Let $(Y^{\T_n}_k)_{k\in \N}$ the simple random walk on $\T_n$ and assume further that $(n^{-1/2}Y^{\T_n}_{n^{3}t})_{t\geq 0}$ scales to the Brownian motion on the CRT.
Now, pick uniformly at random $K$ points on $\T_n$ and consider the tree $\T_n^K$ defined by the geodesics between the root and these $K$ points. Define $\pi_n^K$ to be the natural projection from $\T_n$ to $\T_n^K$. Then, by the convergence of $\T_n$ to the CRT and $Y^{\T_n}$ to the BM on the CRT, we should have that  
\begin{equation}\label{eq:matame}
(n^{-1/2}\pi_{ n}^K(Y^{\T_n}_{ n^{3}t}))_{t\geq 0}\stackrel{d}{\to}(Z^{K\text{-crt}}_t)_{t\geq 0},
\end{equation} where $\stackrel{d}{\to}$ denotes convergence in distribution and the process $Z^{K\text{-crt}}$ is obtained from the BM on the CRT as follows: pick $K$ points on the CRT according to natural uniform measure on the CRT. Let $\frak{T}^{(K)}$ be the geodesic defined by these $K$ points and the root (see Section~\ref{sect_BCRT} for a formal definition). Then $Z^{K\text{-crt}}$ is the projection of the Brownian motion on the CRT onto $\frak{T}^{(K)}$. Moreover, the process $Z^{K\text{-crt}}$ can be expressed as a $\frak{T}^{(K)}$-SSBM. 
Conversely, as we will see in \S\ref{sect_final} the SSBMs can be used to check convergence towards the BM on the CRT. Indeed, by verifying the convergence in \eqref{eq:matame} for all $K\in\N$, we can deduce that the only possible limit of $(n^{-1/2}Y^{\T_n}_{n^{3}t})_{t\geq 0}$ is the BM on the CRT. Therefore, \eqref{eq:matame} for all $K\in\N$ together with tightness, imply that $(n^{-1/2}Y^{\T_n}_{n^{3}t})_{t\geq 0}$ converges to the BM on the CRT.


\subsection{Organization of the paper}

The paper is organized as follows. We begin, in Section~\ref{sect_RTWR}, by recalling the needed convergence results about the class of processes called Randomly trapped random walks (RTRW). The notion of RTRW was introduced in~\cite{rtrw}, as well as their scaling limits, the Spatially Subordinated Brownian Motions (SSBMs). These notions will be important for the proofs of the convergence theorems, Theorems~\ref{thm:IIC} and \ref{thm:IPC}.

The convergence theorems for RTRW depend on two basic sets of assumptions. First we need an assumption, called Assumption L, which is related to the convergence of inverse local times at a fixed vertex of the tree. Second, we need an assumption, called Assumption HT, giving a heavy tail behavior for the mean-time spent in traps. The organization of the paper follows this closely.
In Section~\ref{s:proofoftheorem1}, we prove Theorems~\ref{thm:IIC} and~\ref{t:descriptionofziic} by checking Assumptions HT and L in  \S~\ref{sect_HTIIC} and \S~\ref{sect_LIIC} respectively. The proof of Assumption L depends on our key result, Proposition \ref{prop:annealedlocal}, which states the convergence of local times of random walks in trees to the local time of the Brownian motion in the Continuum Random tree. Section~\ref{sect_local_time} is devoted to the proof of Proposition \ref{prop:annealedlocal}, 
 


In Section~\ref{sect_ipc}, we consider the same questions for the IPC in the same order. We prove Assumption HT in Section~\ref{sect_IPCHT}, Assumption L in Section~\ref{sect_LIPC}, and finally wrap up the proof of Theorems~\ref{thm:IPC} and~\ref{t:descriptionofzipc} in Section~\ref{s:proofoftheorem2}.

In Section~\ref{sect_last} we prove the remaining Theorem \ref{prop:alternativeexrepssionforziic} and Theorem \ref{prop:alternativeexrepssionforzipc} for the IIC and IPC respectively, using an alternative representation of the processes $Z^{\IIC}$ and $Z^{\IPC}$, in terms of the Brownian Motion on the CRT.

Finally in Section~\ref{sect_finite_SSBM}, we discuss the convergence of the random walk projected to the $K$-skeleton towards $\frak{T}^{(K)}$-SSBMs.

\section{Randomly trapped random walks}\label{sect_RTWR}
In this section we will show that $X^\IIC$ and $X^\IPC$ belong to a general class of processes called \textit{Randomly trapped random walks} (RTRW). We will also recall some general convergence results of RTRW which will be used in the proofs of Theorem \ref{thm:IIC} and Theorem \ref{thm:IPC}.

  A RTRW should be regarded as a random walk moving among a random environment composed of traps, where the traps retain the walk for a certain amount of time. Those processes were introduced in \cite{rtrw}, where the one-dimensional case (i.e., when the state space is $\Z$) was studied in detail. In particular, all possible scaling limits on i.i.d.~environments were identified, with some highly non-trivial processes being part of the picture. In fact, a new class of processes, called \textit{Spatially Subordinated Brownian motions} (SSBM) appeared in the limit. As we will see, the scaling limit $Z^{\IIC}$ of Theorem \ref{thm:IIC} falls into that class (if we disregard the unessential difference that the SSBM are defined as taking values in $\R$, whereas $Z^\IIC$ takes values in $\R_+$). The case of the IPC turns out to be very similar, with slight differences coming from the fact that the branches of the IPC are not i.i.d.

To define RTRW, first we have to define the \textit{quenched} versions of those processes, i.e.,~when the environment is non-random. Those quenched versions are called \textit{Trapped random walks} (TRW).
Let $G$ be a graph and $\boldsymbol\pi=(\pi_x)_{x\in G}$ be a family of probability measures on $(0,\infty)$ indexed by the vertices of $G$. Let $(Z[\boldsymbol\pi]_t)_{t\geq0}$ be a continuous-time random walk on the vertices of $G$ which, each time it visits a vertex $x\in G$, it stays there a time distributed according to $\pi_x$ and then jumps to one of its nearest neighbors chosen uniformly at random. If $Z[\boldsymbol\pi]$ visits $x$ again, the duration of the new visit is sampled independently of the duration of the previous visits. The process $Z[\boldsymbol\pi]$ is a Trapped random walk with \textit{trapping landscape} $\boldsymbol\pi$.

The Randomly trapped random walks are obtained by adding an extra layer of randomness, i.e., by considering TRW on random trapping landscapes.
Let $M_1(\R_+)$ be the space of probability measures on $\R_+$ endowed with the topology of weak convergence and $M_1(M_1(\R_+))$ be the space of probability measures on $M_1(\R_+)$.
Let $\frak{P}\in M_1(M_1(\R_+))$ and $\boldsymbol{\pi}=(\pi_x)_{x\in G}$ be an i.i.d.~family of random probability measures distributed according to $\frak{P}$ defined on a probability space $(\mathcal{X},\mathcal{G},\boldmath{P})$.
We say that the process $Z[\boldsymbol\pi]$ is a Randomly Trapped Random Walk with an i.i.d.~trapping landscape $\boldsymbol\pi$. To include the case of the IPC, we also need to consider processes defined on environments which are not i.i.d. Let $\boldsymbol\pi=(\pi_x)_{x\in G}$ be a random trapping landscape, i.e., $\boldsymbol\pi$ is a random object taking values in $M_1((0,\infty))^G$ defined on a probability space $(\mathcal{X},\mathcal{G},\boldmath{P})$. The random walk $Z[\boldsymbol\pi]$ is called Randomly Trapped Random Walk (RTRW) with trapping landscape $\boldsymbol\pi$. For each $\omega\in\mathcal{X}$, we denote by $\boldmath{P}_{\omega}$ the law of $Z[\boldsymbol{\pi}]$ for a fixed realization of ${\boldsymbol\pi}(\omega)$ of the environment. The annealed law is defined as the semi-direct product $\Pb:=\boldmath{P}\times\boldmath{P}_\omega$.

Now, we aim to express $X^\IIC$ and $X^{\IPC}$ as RTRW.
Let $I$ be a rooted tree with root $\rho$ and $(Y_k)_{k\in\N_0}$ be a discrete-time, nearest neighbor, symmetric random walk on $I$ starting at the root. Define \begin{equation}\label{eq:sigma}
\sigma[I]:=\min\{l>0:Y_l=\rho\}.
\end{equation}
 Let $\tilde{I}$ be the tree obtained from $I$ by attaching two extra vertices $v_1,v_2$ to the root and $(\tilde{Y}_k)_{k\in\N_0}$ be a discrete-time, symmetric random walk on $\tilde{I}$ started at the root. Define
  \begin{equation}\label{eq:sigmatilde}
  \tilde{\sigma}[I]:=\min\{l>0:\tilde{Y}_l\in\{v_1,v_2\}\}.
  \end{equation}
   We denote by $\nu[I],\tilde{\nu}[I]$ the distribution of $\sigma[I],\tilde{\sigma}[I]$ respectively. Using this notation we can express $X^{\IIC}$ as a RTRW in $\N$ with random trapping landscape $\boldsymbol{\pi}^\IIC:=(\tilde{\nu}[\B_x])_{x\in\N}$ (we recall that $(\B_k)_{k\in\N}$ are the branches of the IIC). Similarly $X^{\IPC}$ is a RTRW in $\N$ with random trapping landscape $\boldsymbol\pi^\IPC:=(\tilde{\nu}[\Lv_x])_{x\in\N}$. The trapping landscape of $X^{\IIC}$ is i.i.d., because the branches $(\B_x)_{x\in\N}$ are i.i.d. Note, however, that this is not true for $X^{\IPC}$ because, as we have previously said, the branches $(\Lv_x)_{x\in\N}$ are not i.i.d.

Now we prepare the ground for the definition of the Spatially subordinated Brownian motions.
Let $\mathfrak {F}$ be the set of Laplace exponents of
  subordinators, that is, $\mathfrak{F}$ is the set of continuous functions
  $f:\mathbb R_+\mapsto \mathbb R_+$ that can be expressed as
  \begin{equation}
    \label{e:fdPi}
    f(\lambda)=f_{\mathtt d, \Pi }(\lambda )
    := \mathtt d\lambda+\int_{\mathbb{R}_+}(1-e^{-\lambda t})\Pi(dt)
  \end{equation}
  for $\mathtt d\ge 0$ and a measure $\Pi$ satisfying
  $\int_{(0,\infty)}(1\wedge t)\Pi(dt)< \infty$.

The definition of the SSBM will depend on two parameters, $\gamma\in(0,1)$ and $\mathbb{F}\in M_1(\frak{F})$, where $M_1(\frak{F})$ denotes the space of probability measures on $\frak{F}$.
Let $(V^{\gamma}_t)_{t\in\R}$ be a two-sided $\gamma$-stable subordinator. That is, $V^\gamma$ is the Subordinator characterized by
\[\E[\exp(-\lambda V^\gamma_t)]=e^{-t\int_{\R_+}(1-e^{-\lambda x})\gamma x^{-1-\gamma}dx}.\]
  It is a known fact that $V^{\gamma}$ is a pure jump process and therefore its corresponding Lebesgue-Stieltjes random measure $\mu$, defined by $\mu(a,b]=V^{\gamma}_b-V^{\gamma}_a$, can be expressed as $\mu:=\sum_{i\in\N}y_i\delta_{x_i}$. Furthermore, it is also known that the collection of points $(x_i,y_i)_{i\in\N}$ is distributed as an inhomogeneous Poisson point process in $\R\times\R_+$ with intensity measure $\gamma y^{-1-\gamma}dydx$.

   Also, let $(f_i)_{i\in\N}$ be an i.i.d.~family of Laplace exponents sampled according to $\mathbb{F}$ and independent of $\mu$. Let $((S^i(t))_{t\geq0})_{i\in\N}$ be an independent sequence of subordinators with Laplace exponents $(f_i)_{i\in\N}$. Finally, let $(B_t)_{t\geq0}$ be a one-dimensional, standard Brownian motion started at the origin independent of everything else and $l(x,t)$ be its local time.
Define
\[\phi_t:=\sum_{i\in\N}y_i^{1+\alpha}S^i(y_i^{-\alpha}l(x_i,t))\]
and
\[\psi_t:=\inf\{s\geq0:\phi_s>t\}.\]
The Spatially Subordinated Brownian motion\footnote{The definition of SSBM given in \cite{rtrw} is slightly more general to the one presented here. Nevertheless, all the SSBMs appearing in this article fall under this definition.} is the process defined as \[B^{\mathbb{F},\gamma}_t:=B_{\psi_t}.\] Note that $Z^{\IIC}$ corresponds to an SSBM where $\gamma=1/2$ and $\mathbb{F}$ is the law of the random Laplace exponent of the inverse local time at the root of the Brownian motion in the \emph{Continuum Random tree}.

Let $Z[\boldsymbol\pi]$ be a RTRW on an i.i.d.~trapping landscape $\boldsymbol\pi=(\pi_x)_{x\in\Z}$ with marginal $\frak{P}\in M_1(M_1(\R_+))$.  In \cite{rtrw}, there are given criteria under which $Z[\boldsymbol\pi]$ converges to an SSBM. That convergence result will be one of the main tools to prove Theorems \ref{thm:IIC}, and \ref{t:descriptionofziic}, so we  proceed to recall it.

 Let $m:M_1(\R_+)\to[0,\infty]$ be defined as
\[
m(\pi):=\int_{\R_+}t\pi(dt).
\]
That is, $m(\pi)$ is the mean of the probability distribution $\pi$. Our first assumption is that the distribution of $m(\pi)$ has heavy tails.
\begin{assumption(pp)}
 There exists $\gamma\in(0,1)$ and $c> 0$ such that
\[
\lim_{u\to\infty}u^{\gamma}\frak{P}[\pi \in M_1(\R_+):m(\pi)>u]=c.
\]
\end{assumption(pp)}

Now, we turn our attention to the statement of the second assumption. Define
\[d(\epsilon):=c^{1/\gamma}\epsilon^{-1/\gamma}\qquad \text{and }\qquad q(\epsilon):=\frac{\epsilon}{d(\epsilon)}.\]
To understand the role of $d(\epsilon)$ and $q(\epsilon)$, imagine a RTRW after $\epsilon^{-2}$ steps. Its range will be of order $\epsilon^{-1}$ and, in view of Assumption HT, we have that $d(\epsilon)\sim\max_{x\in[-\epsilon^{-1},\epsilon^{-1}]}m(\pi_x)$ is the order of magnitude of the deepest trap found by the walker. Since (after $\epsilon^{-2}$ steps) the RTRW has visited the deepest trap about $\epsilon^{-1}$ times, the scale $q(\epsilon)=\epsilon^{-1}\times d(\epsilon)$ represents the order of magnitude of the accumulated time spent in the deepest trap. 
As is typical for heavy tailed random variables, the time spent in the deepest trap is of the same order of magnitude as the total time spent in all the traps, so $q(\epsilon)$ also is the scale of the time that it takes for the RTRW to give $\epsilon^{2}$ steps.

For a probability measure $\nu\in M_1(\R_+)$, let $\hat{\nu}$ denote the Laplace transform of $\nu$. Let $\Psi_\epsilon:M_1(\R_+)\to \mathfrak {F}$ be defined as
\begin{equation}\label{eq:psiepsilon}
\Psi_\epsilon(\nu)(\lambda):=\epsilon^{-1}(1-\hat{\nu}(q(\epsilon)\lambda)).
\end{equation}
Indeed, $\Psi_\epsilon(\nu)$ is the Laplace exponent of a compound Poisson process of intensity $\epsilon^{-1}$ and whose jump distribution is $\nu$, scaled by a factor $q(\epsilon)$. Heuristically, (the subordinator associated to) $\Psi_\epsilon(\nu)$ gives the rescaled, accumulated time spent in a trap of distribution $\nu$.
\begin{assumption(d)}
For each $a>0$, let $\pi^a$ be a random measure having the distribution of $\pi_0$ conditioned on $m(\pi_0)=a$. Then
\[
\textrm{law of } \Psi_{\epsilon}(\pi^{d(\epsilon)})\stackrel{\epsilon\to0}{\to}\mathbb{F}_1,
\]
for some $\mathbb{F}_1\in M_1(\frak{F})$ non-trivial, that is $\mathbb{F}_1\neq\delta_{\lambda\mapsto 0}$, where $\frak{F}$ is endowed with the topology of uniform convergence over compacts.
\end{assumption(d)}
\begin{remark}
The space $C(\R_+)$, and thus $\mathfrak{F}\subset C(\R_+)$, endowed with the topology of uni- form convergence over compact sets is separable. It is a known fact that in the space $\mathfrak{F}$ the pointwise convergence and the uniform convergence over compact sets coincide. (Recall $\mathfrak{F}$ is the space of Laplace exponents. When the Laplace exponents converge pointwise to an element of $\mathfrak{F}$, the corresponding probability measures converge weakly, which in turns gives the uniform convergence over compacts.) 
\end{remark}

We are ready to state the convergence result (Theorem 2.16 in \cite{rtrw}):
\begin{theorem}\label{prop:iidrtrw}
Suppose $Z[\boldsymbol\pi]$ is an i.i.d.~RTRW for which assumptions HT and L hold. Then, as $\epsilon\to0$, we have that $(\epsilon Z[\pi]_{q(\epsilon)^{-1}t})_{t\geq0}$ converges in $\Pb$-distribution to $(B^{\mathbb{F}_1,\gamma}_t)_{t\geq 0}$ on $(D(\R_+)$ endowed with the Skorohod $J_1$ topology.
\end{theorem}
For the definition of the $J_1$ topology we refer to \cite[\S3.3]{whi02}.
 The strategy to prove Theorems \ref{thm:IIC} and \ref{t:descriptionofziic} is to verify assumptions HT and L for $X^{\IIC}$ and to apply the theorem above.

 Now we turn our attention to the case of the IPC. The techniques developed in \cite{rtrw} also yield an analog of Theorem \ref{prop:iidrtrw} which is suitable to treat some non-i.i.d.~RTRW. We will make use of that result in the proof of Theorem \ref{thm:IPC}, so we proceed to recall it.

Let $Z[\boldsymbol\pi]$ be a RTRW with random trapping landscape $\boldsymbol\pi=(\pi_x)_{x\in\Z}$. We assume that there exists a family of probability distributions $(P_a)_{a>0}\subset M_1(M_1(\R_+))$ such that, conditioned on $(m(\pi_x))_{x\in\Z}=(m_x)_{x\in\Z}$, $\boldsymbol\pi$ is distributed according to $\otimes_{x\in\Z}P_{m_x}$. In other words, the random measures $(\pi_x)_{x\in\Z}$ are independent when conditioned on the depths $(m(\pi_x))_{x\in\Z}$.

Define $V\in D(\R)$ as
\[
V_x:=\begin{cases}
            \sum_{i=1}^{\lfloor x \rfloor} m(\pi_i) &:x\geq1,\\
            0                                                      &:x\in[0,1),\\
            -\sum_{i=\lfloor x \rfloor}^0 m(\pi_i)  &:x<0.
         \end{cases}
\]
The analogous assumptions are the following.
 \begin{assumption(pp*)} There exists $\gamma\in(0,1)$ such that $(\epsilon^{1/\gamma}V_{\epsilon^{-1}x})_{x\in\R}$ converges in distribution on $(D(\R),J_1)$ to a strictly increasing, pure-jump process $(V_x^0)_{x\in\R}$.
 \end{assumption(pp*)}
 
The condition $\widetilde{\text{HT}}$ looks different from HT but they are actually similar since the heavy-tailed condition $\text{HT}$ implies the convergence of a rescaled process towards a stable subordinator which is a strictly increasing, pure-jump process.

\begin{assumption(d*)} Let $\pi^a$ be a random measure having law $P_a$. Then
\[
\textrm{law of } \Psi_\epsilon(\pi^{d(\epsilon)})\stackrel{\epsilon\to0}{\to}\mathbb{F}_1
\]
for some $\mathbb{F}_1\in M_1(\frak{F})$ non-trivial, that is $\mathbb{F}_1\neq\delta_{\delta\mapsto 0}$, where $\frak{F}$ is endowed with the topology of uniform convergence over compacts.
\end{assumption(d*)}

Now we define a class of processes which corresponds to our extension of the notion of SSBM which appears as scaling limits of RTRW satisfying assumptions $\widetilde{\text{HT}}$ and $\tilde{\text{L}}$.

Let $\gamma$ and $V_0$ be as in Assumption $\widetilde{\text{HT}}$ and $\mathbb{F}_1$ be as in Assumption $\tilde{\text{L}}$. Let $\nu:=\sum_{i\in\N}y_i\delta_{x_i}$ be the random Lebesgue-Stieltjes measure associated with $V^0$ and $(f_i)_{i\in\N}$ be an i.i.d.~family of Laplace exponents distributed according to $\mathbb{F}_1$ and independent of $V^0$. Let $((S^i(t))_{t\geq0})_{i\in\N}$ be an independent sequence of subordinators with Laplace exponents $(f_i)_{i\in\N}$. Also, let $(B_t)_{t\geq0}$ be a one-dimensional, standard Brownian motion started at the origin independent of everything else and $l(x,t)$ be its local time.
Define
\[\phi_t:=\sum_{i\in\N}y_i^{1+\gamma}S^i(y_i^{-\gamma}l(x_i,t))\]
and
$\psi_t:=\inf\{s\geq0:\phi_s>t\}$.
Finally, define \[B^{\mathbb{F}_1,V^0,\gamma}_t:=B_{\psi_t}.\]
Observe that $X^{\IPC}$ corresponds to taking $\gamma=1/2$, $V^0_x=\mu_{\IPC}(0,x]$ where $\mu_{\IPC}$ is as in \eqref{ipclimitmeasure} and $\mathbb{F}_1$ as the law of the random Laplace exponent of the inverse local time at the root of the Brownian motion in the \emph{Continuum Random tree}.

The next proposition states convergence of RTRW satisfying assumptions $\widetilde{\text{HT}}$ and $\tilde{\text{L}}$ to the processes defined above

\begin{theorem}\label{RTRWIPC}Suppose $Z[\boldsymbol\pi]$ is a (non-necessarily i.i.d.) randomly trapped random walk for which assumptions $\widetilde{\text{HT}}$ and $\tilde{\text{L}}$ hold. Then, as $\epsilon\to0$, we have that $(\epsilon Z[\boldsymbol\pi]_{q(\epsilon)^{-1}t})_{t\geq0}$ converges in $\Pb$-distribution to $(B^{\mathbb{F}_1,V^0,\gamma}_t)_{t\geq 0}$ on $(D(\R_+),J_1)$.
\end{theorem}

 This proposition can be proved by following exactly the same arguments of Theorem 2.13 in \cite{rtrw}.

\section{Convergence results for the IIC: proof of Theorems \ref{thm:IIC} and \ref{t:descriptionofziic}}\label{s:proofoftheorem1}

The proof will consist in verifying the assumptions HT and L of Theorem \ref{prop:iidrtrw}. Assumption HT is proved in Lemma \ref{lem:assumptionhtforiic} below. In this case, the parameter $\gamma$ equals $1/2$ which, according to Theorem \ref{prop:iidrtrw}, yields that the scaling exponent $q(\epsilon)$ is of order $\epsilon^{1+\gamma}=\epsilon^{3}$, which explains why $1/3$ is the spatial sub-diffusivity exponent in Theorem \ref{thm:IIC}.
Assumption $HT$ is proved in Lemma \ref{assumptionL} below, using the convergence of local times of Proposition \ref{prop:annealedlocal}.

 Since Theorem \ref{prop:iidrtrw} deals with processes defined in the whole axis and $X^{\IIC}$ is defined in the positive part of the axis, we introduce an analog of $X^{\IIC}$ which is defined in the whole axis.
Let $(\B_x)_{x\in\Z}$ be an i.i.d.~family of critical percolation clusters on $\mathbb{T}_2^\ast$ and $(X^{\IIC\ast}_t)_{t\geq0}$ be an i.i.d.~Randomly trapped random walk with $(\tilde{\nu}[\B_x])_{x\in\Z}$ as its random trapping landscape.

Now we proceed to define the RTRW which appears as scaling limit of $X^{\IIC\ast}$. Since RTRW are parametrized by the parameters $\gamma$ and $\mathbb{F}_1$. Since we have already anticipated that $\gamma=\frac{1}{2}$, we focus on the construction of the Laplace exponent distribution $\mathbb{F}_1$.
  Let $(W_t)_{t\in[0,1]}$ be a normalized Brownian excursion defined in a probability space $(\mathcal{X},\mathcal{G},P)$. Let $w\in\W$ be a realization of $W$ and $\mathfrak{T}_w$ be its corresponding $\R$-tree (see \S \ref{s:crt} below for the definition of the correspondence), so that $\mathfrak{T}_w$, when regarded as a random object defined over $(\mathcal{X},\mathcal{G},P)$, has the law of the \emph{Continuum Random tree} (CRT) (see Definition \ref{definitioncrt} below). As we will see in Section \ref{s:BMCRT}, it is possible to define, for almost every $w\in\mathcal{X}$, the \emph{Brownian motion} $(B^{\mathfrak{T}_w}_t)_{t\geq0}$ in $\mathfrak{T}_w$ and its corresponding local time $(L_w(x,t))_{x\in\mathfrak{T}_w,t\geq0}$. Let $\rho$ be the root of the CRT. By virtue of the strong Markov property of $B^{\mathfrak{T}_w}$ and the fact that, for all $t\in\R_+$, $\inf\{s\geq0:L_w(\rho,s)\geq t\}$ is a stopping time, we have that the inverse local time $(L^{-1}_w(\rho,t))_{t\geq0}$ has independent and stationary increments, that is, $L^{-1}_w(\rho,\cdot)$ is a subordinator.  
  Let $f^w$ be the Laplace exponent of $(L^{-1}_w(\rho,t))_{t\geq0}$ and $\mathbb{F}^\ast_1\in M_1(\frak{F})$ be defined as $\mathbb{F}_1^\ast[A]:=P[f^W\in A]$ for each $A$ Borelian of $\frak{F}$.
  From standard considerations about Laplace exponents, $\pi^{1/2} f^w(\pi^{-1/2}\cdot)$ is the Laplace exponent of of $(\pi^{-1/2}L^{-1}_w(\rho,\pi^{1/2}t))_{t\geq0}$. Let $\mathbb{F}_1\in M_1(\frak{F})$ be the law of $\pi^{-1/2} f^w(\pi^{1/2}\cdot)$. 

We will prove the following
\begin{proposition}\label{convergenceofWIICast}
 $(\epsilon X^{\IIC\ast}_{\pi^{-1}\epsilon^{-3}t})_{t\geq0}$ converges in distribution to the randomly subordinated Brownian motion $(B^{\mathbb{F}_1,1/2}_t)_{t\geq 0}$ on $(D(\R_+),J_1)$.
\end{proposition}

\subsection{Assumption HT for $X^{\IIC\ast}$}\label{sect_HTIIC}
In this subsection we will prove that assumption HT holds for $X^{\IIC\ast}$. We recall that, for each $\nu\in M_1(\R_+)$, $m(\nu)$ stands for $\int_{\R_+}t\nu(dt)$.
\begin{lemma}\label{lem:assumptionhtforiic} We have that
 \[\lim_{u\to\infty}u^{1/2}\Pb[m(\tilde{\nu}[\B_0])>u]=\pi^{-1/2},\]
 where $\tilde{\nu}[\B_0]$ is as in \eqref{eq:sigmatilde}.
\end{lemma}
It follows directly from the lemma above that the scaling functions $d(\epsilon)$ and $q(\epsilon)$ of assumption L equal \begin{equation}\label{eq:scales}
d(\epsilon)=\pi^{-1}\epsilon^{-2}\qquad q(\epsilon)=\pi \epsilon^{3}.
\end{equation}
In the proof we will use the following.
For a r.v. $\mathrm{X}$, let $\hat{\mathrm{X}}$ denote its Laplace transform.
\begin{lemma}\label{laplacetransformofcardinality}
Let $N_p$ be the cardinality of the connected component of the root $\rho$ of $\mathbb{T}_2^\ast$ under percolation of parameter $p\leq1/2$, then
\[\hat{N}_p(\lambda)=\frac{1-\sqrt{1-4p(1-p)\exp(-\lambda)}}{2p}.\]
\end{lemma}
\begin{proof}[Proof of Lemma \ref{laplacetransformofcardinality}]
First we compute the Laplace transform of $N_p^\ast$ which is the size of a percolation cluster on $\mathbb{T}_2$ with parameter $p$.
By conditioning on the status of the edges emerging from the root we find that
\[\hat{N}_p^\ast(\lambda)=\exp(-\lambda)[(1-p)+p\hat{N}_p^\ast(\lambda)]^2 \]
Therefore
\[\hat{N}_p^\ast(\lambda)=\frac{1-2p(1-p)\exp(-\lambda)-\sqrt{1-4p(1-p)\exp(-\lambda)}}{2 p^2\exp(-\lambda)}\]
where the solution
 \[\frac{1-2p(1-p)\exp(-\lambda)+\sqrt{1-4p(1-p)\exp(-\lambda)}}{2 p^2\exp(-\lambda)}\]
 has been discarded because, when $p<1/2$, it yields that $\hat{N}_p^\ast(0)>1$ and when $p=1/2$, it yields that
\[\hat{N}_{1/2}^\ast(\lambda)-1=2\exp(\lambda)(1-\exp(-\lambda)+\sqrt{1-\exp(-\lambda)})\] which is positive when $\lambda>0$.

Again, conditioning on the status of the edge of the root of $\mathbb{T}_2^\ast$ we find that
\[\hat{N}_p(\lambda)=\exp(-\lambda)(p\hat{N}_p^\ast(\lambda)+1-p).\]
Therefore
\[
\hat{N}_p(\lambda)=\frac{1-\sqrt{1-4p(1-p)\exp(-\lambda)}}{2p}.
\]
\end{proof}

\begin{proof}[Proof of Lemma \ref{lem:assumptionhtforiic}]
We recall that for any rooted tree $I$, $\E[\tilde{\theta}[I]]=|I|$ (see \cite[Lemma 2.28]{Kesten}). In particular $m(\tilde{\nu}[\B_0])=N_{1/2}$.
By Lemma \ref{laplacetransformofcardinality} we have that \[1-\hat{N}_{1/2}(\lambda)=\sqrt{1-\exp(-\lambda)}\sim \sqrt{\lambda}\] as $\lambda\to0$. Therefore applying the Tauberian Theorem (see \cite[Chapter XIII.5, Example(c)]{fel71}) we get that \[\Pb[m(\tilde{\nu}[\B_0])>u]\sim\Gamma(1/2)^{-1}u^{-1/2}=\pi^{-1/2}u^{-1/2}\] as $u\to\infty$, where $\Gamma$ denotes the Gamma function.
\end{proof}

\subsection{Assumption L for $X^{\IIC\ast}$}\label{sect_LIIC}
Here we will prove that assumption L holds for $X^{\IIC\ast}$. Let $\B^n$ be a critical percolation cluster on $\mathbb{T}_2^\ast$ conditioned on having $n$ vertices, were we recall that $\mathbb{T}_2^\ast$ is a regular tree in which each vertex has degree $3$ except for the root which has degree $1$.
Recall that $\mathbb{F}_1^\ast$ is the law of the random Laplace exponent $f^W$ of the inverse local time at the root of the Brownian motion in the CRT.
\begin{lemma}\label{assumptionL}
 \[
\textrm{Law of }\Psi_{n^{-1/2}}(\tilde{\nu}[\B^n])\stackrel{n\to\infty}{\to}\mathbb{F}_1^\ast
\]
\end{lemma}
Recalling the scaling functions $d(\epsilon),q(\epsilon)$  from \eqref{eq:scales} we get the following corollary  of the Lemma above.

\begin{corollary}\label{cor:assumptionL}
\[\textrm{Law of }\Psi_{\epsilon}(\tilde{\nu}[\B^{d(\epsilon)}])\stackrel{\epsilon\to0}{\to} \mathbb{F}_1^\ast.\]
\end{corollary}

\begin{proof}[Proof of Lemma \ref{assumptionL}]
 Since the degree of the root $\rho$ on $\mathbb{T}_2^\ast$ is $1$, $\B^n-\{\rho\}$ has the law of a critical percolation cluster on $\mathbb{T}_2$ conditioned on having $n-1$ vertices. Hence $\B^n-\{\rho\}$ can be seen as a Galton-Watson tree whose offspring distribution is Binomial of parameters $N=2,p=1/2$, conditioned on having $n-1$ vertices.

  Let $\tilde{v}^n$ and $v^n$ denote respectively the \emph{depth-first search} and \emph{search-depth} processes of $\B^n-\{\rho\}$ (see \S \ref{s:crt} below for the formal definitions).  By virtue of \cite[Theorem 23]{AldousCRT3}, we have that
 \begin{equation}\label{eq:preconvergenceofsearchdepthprocesses}
 \left(n^{-1/2}v^{n+1}(t)\right)_{t\in[0,1]}\stackrel{d}{\Rightarrow}\left(\sqrt{2}W_t\right)_{t\geq0}\quad \text{as } n\to\infty
 \end{equation}
 on $C[0,1]$ endowed with the uniform topology, where $(W_t)_{t\in[0,1]}$ is the normalized Brownian excursion.

 Let $\tilde{w}^n$ and $w^n$ be the depth-first search and search-depth processes of $\B^n$ respectively. Since $\tilde{w}^n(i)=\tilde{v}^n(i-1)$, $i=2,\dots 2n-2$ we have that
  \[d(\tilde{w}^n(i),\tilde{v}^n(i))=1,\]
  for all $i=2,\dots,2n-2$. That, together with display \eqref{eq:preconvergenceofsearchdepthprocesses} imply that
 \begin{equation}\label{eq:convergenceofsearchdepthprocesses}
 \left(n^{-1/2}w^{n}(t)\right)_{t\in[0,1]}\stackrel{d}{\Rightarrow}\left(\sqrt{2}W_t\right)_{t\geq0}\quad \text{as } n\to\infty,
 \end{equation}
 in the uniform topology.

  By virtue  of the Skorohod representation theorem and display \eqref{eq:convergenceofsearchdepthprocesses} we can find coupled processes $\bar{w}^n,n\in\N$  and $\bar{W}$ defined on a common probability space $(\Omega,\mathcal{F},\mathbb{Q})$ such that $\bar{w}^n$ is distributed as $w^n$, $\bar{W}$ is distributed as $W$ and
 \begin{equation}\label{eq:convergenceofcoupledsearchdepthprocesses}
 \left(n^{-1/2}\bar{w}^n_t\right)_{t\in[0,1]} \stackrel{u}{\to} \left(\sqrt{2}\bar{W}_t\right)_{t\in[0,1]},\quad \mathbb{Q}\text{-a.s.,}
 \end{equation}
 where $\stackrel{u}{\to}$ denotes uniform convergence.

 For every $n\in\N$, let $\bar{\B}^n$ be the random tree with $\bar{w}^n$ as its search-depth process. The trees $\bar{\B}^n$ are well defined because any ordered, rooted tree can be reconstructed from its search-depth process (see \S \ref{s:crt} below). 
   Let $(\mathrm{T}_n)_{n\in\N}$ be a fixed realization of the random trees $(\bar{\B}_n)_{n\in\N}$ and $w$ be a realization of the Brownian excursion $(\bar{W}_t)_{t\in[0,1]}$. Let $l^{-1}[\mathrm{T}_n]$ be the \emph{inverse local time at the root} of a RW on $\mathrm{T}_n$ (see \eqref{eq:definitionofdiscretelocaltime} below)
 Using Corollary \ref{quenchedconverenceofinverselocaltimes} we have that, $\mathbb{Q}$-a.s.
  \[
  \left(n^{-3/2}l^{-1}[\mathrm{T}_n]_{n^{1/2}t}\right)_{t\geq0}\Rightarrow\left(L^{-1}_{\sqrt{2}w}(\rho,2t)\right)_{t\geq0},
    \]
  in distribution in $(D(\R_+),M_1)$, $\mathbb{Q}$-almost surely.
   By Lemma \ref{l:irrelevanceoftheconstant} below, we can replace $L^{-1}_{\sqrt{2}w}$ by $L^{-1}_{w}$ in the display above to get that, $\mathbb{Q}$-almost surely,  
    \[
     \left(n^{-3/2}l^{-1}[\mathrm{T}_n]_{n^{1/2}t}\right)_{t\in[0,T]}\Rightarrow\left(L^{-1}_{w}(\rho,2t)\right)_{t\in[0,T]},\]
 in distribution in $(D(\R_+),M_1)$.
  It is a known fact that convergence in the $M_1$ topology implies convergence of single-time distributions at continuity points of the limiting function. Therefore, since every point (in particular $t=1$) is almost surely a continuity point of $L^{-1}_{\bar{W}}$, we have that, for all $t\geq0$, $\mathbb{Q}$-almost surely,
  \[n^{-3/2}l^{-1}[\mathrm{T}_n]_{n^{1/2}}\Rightarrow L^{-1}_{w}(\rho,2)\]
  in distribution.
  On the other hand, since convergence in distribution of random variables implies convergence of the respective Laplace transforms and $n^{-3/2}l^{-1}[\mathrm{T}_n]_{n^{1/2}}$ is the sum of $n^{1/2}$ i.i.d. random variables distributed as $\nu[\mathrm{T}_n]$ scaled by $n^{-3/2}$, we have that 
   \[\log([\hat{\nu}[\mathrm{T}_n](n^{-3/2}\lambda )]^{n^{1/2}})\to \log\left(E\left[\exp{(-\lambda L^{-1}_w(\rho,2))} \right]\right) = 2f^w(\lambda),\]
    $\mathbb{Q}$-a.s., for all $\lambda\geq0$ as $n\to\infty$, where $E$ denotes expectation.
     Therefore,
    \begin{equation}\label{eq:preassumptionL}
     n^{1/2}(1-\hat{\nu}[\mathrm{T}_n](n^{-3/2}\lambda ))\to 2f^{w}(\lambda)
     \end{equation}
    $ \mathbb{Q}\text{-a.s.,}$ for all $\lambda\geq0$ as $n\to\infty$.

We would like to have a convergence result as \eqref{eq:preassumptionL} but with $\tilde{\nu}[\mathrm{T}_n]$ instead of $\nu[\mathrm{T}_n]$.
For each $n\in\N$, let $(H^n_k)_{k\in\N}$ be an i.i.d.~sequence of random variables distributed according to $\nu[\mathrm{T}_n]$. For all $n\in\N$ we set $S^n_0:=0$ and
 \[S^n_t:=\sum_{i=1}^{\lfloor t \rfloor} H^n_i \qquad t \geq0 .\]
 Observe that, for all $n\in\N$, $S^n$ is distributed as the inverse local time of a random walk on $\mathrm{T}_n$. 

Let $G$ be a geometric random variable of parameter (probability of success) $1/3$ independent of the $H^n_i,i\in\N$.
Then
 \[\hat{\tilde{\nu}}[\mathrm{T}_n](\lambda)=\E\left[\exp\left(-\lambda\left(1+\sum_{i=0}^{G}H^n_i\right)\right)\right]\]
\[=\exp(-\lambda)\frac{2}{3-\hat{\nu}[\mathrm{T}_n](\lambda)}.\]
 Therefore
 \[1-\hat{\tilde{\nu}}[\mathrm{T}_n](\lambda)\sim \frac{1-\hat{\nu}[\mathrm{T}_n](\lambda)}{2}\]
as $\lambda\to0$. This, together with display \eqref{eq:preassumptionL} imply that
\[n^{1/2}(1-\hat{\tilde{\nu}}[\mathrm{T}_n](n^{-3/2}\lambda))\to f^{w}(\lambda),\qquad \mathbb{Q}\text{-a.s.,}\]
for all $\lambda\geq0$.
\end{proof}
\begin{proof}[Proof of Proposition \ref{convergenceofWIICast}]
Follows directly from Lemma \ref{lem:assumptionhtforiic}, Corollary \ref{cor:assumptionL} and Theorem \ref{prop:iidrtrw}. 
\end{proof}
\subsection{Proof of Theorems \ref{thm:IIC} and \ref{t:descriptionofziic}}\label{sect_zzz}
Theorems \ref{thm:IIC} and \ref{t:descriptionofziic} are essentially equivalent to Proposition \ref{convergenceofWIICast} but with processes restricted to the positive axis.
 Let \[\sigma^\epsilon_t:=\int_0^t 1_{(0,\infty)}(\epsilon X^{\IIC\ast}_{\pi^{-1}\epsilon^{-3}s}>0)ds\]
and 
\[\sigma:=\int_0^t1_{(0,\infty)}(B^{\mathbb{F}_1^\ast,1/2}_s>0) ds.\]
Let $\tau^\epsilon_t:=\inf\{s\geq 0:  \sigma^\epsilon_s>t\}$ and $\tau_t:=\inf\{s\geq 0:\sigma_s>t\}$. It is clear that $(X^{\IIC\ast}_{\tau^{\epsilon}_t})_{t\geq0}$ is distributed as $(X^{\IIC}_{\epsilon^{-1}t})_{t\geq0}$. On the other hand, using Proposition \ref{convergenceofWIICast} it can be shown that $\sigma^\epsilon\to\sigma$ uniformly over compact sets, 
    in particular, they converge in the weaker $M_1$ topology.
Therefore, by \cite[Theorem 13.2.3]{whi02} 
we deduce that 
\begin{equation}\label{eq:dispabove1}
(\epsilon X^{\IIC}_{\pi^{-1}\epsilon^{-3}t})_{t\geq0}\to(B^{\mathbb{F}_1,{\scriptscriptstyle\frac{1}{2}}}_{\tau_t})_{t\geq0}
\end{equation}
in distribution in the $M_1$ topology. 
By \eqref{eq:dispabove1}, we see that in order to prove the theorem, we need to show that $(B^{\mathbb{F}_1,{\scriptscriptstyle\frac{1}{2}}}_{\tau_{\pi t}})_{t\geq0}$ is distributed as $(Z^{\IIC}_t)_{t\geq0}$. 

It is easy to see that  $(B^{\mathbb{F}_1,{\scriptscriptstyle\frac{1}{2}}}_{\tau_t})_{t\geq0}$ can be constructed exactly as  $(B^{\mathbb{F}_1,{\scriptscriptstyle\frac{1}{2}}}_{t})_{t\geq0}$ but where the underlying process is a reflected Brownian motion instead of a Brownian motion. More precisely, let $(B^+_t)_{t\geq0}$ be a reflected Brownian motion and $l^+(x,t)$ its local time. Let $(x_i,y_i)_{i\in\N}$ be a Poisson point process of intensity $1/2 y^{-3/2} dydx$ independent of $B^+$ and let $(S^i)_{i\in\N}$ be an i.i.d.~sequence of subordinators (and independent of $B^+,(x_i,y_i)_{i\in\N}$) whose Laplace exponents are distributed according to $\mathbb{F}_1$. Let 
\[\phi^+_t:=\sum_{i\in\N} y_i^{3/2} S^i(y_i^{-1/2}l^+(x_i,t))\]
and
\[\psi^+_t:=\inf\{s\geq0:\phi^+_s>t\}.\]
Then we have that $(B^{\mathbb{F}_1,{\scriptscriptstyle\frac{1}{2}}}_{\tau_t})_{t\geq0}$ is distributed as $B^+_{\psi^+_t}$ and $(B^{\mathbb{F}_1,{\scriptscriptstyle\frac{1}{2}}}_{\tau_{\pi t}})_{t\geq0}$ is distributed as $B^+_{\psi^+_{\pi t}}$. It is easy to see that 
\begin{equation}\label{eq:dispabove2}
\begin{aligned}
&\psi^+_{\pi t}=\pi^{-1}\inf\{s\geq0: \phi_s^+>s\}=\pi^{-1}\sum_{i\in\N} y_i^{3/2}S^i(y_i^{-1/2}l^+(x_i,t))\\
&=\sum_{i\in\N} (\pi^{-1}y_i)^{3/2} \pi^{1/2} S^i(\pi^{-1/2}(\pi^{-1}y_i)^{-1/2}l(x_i,t)).
\end{aligned}  
\end{equation}
Since the subordinators $S^i$ are chosen according to $\mathbb{F}_1$, it follows that the subordinators $\pi^{1/2} S^i(\pi^{-1/2}\cdot)$ are chosen according to $\mathbb{F}_1^\ast$. Moreover, since 
$(x_i,y_i)_{i\in\N}$ has intensity $1/2y^{-3/2}dy dx$, we have that $(x_i,\pi^{-1}y_i)_{i\in\N}$ has intensity $1/2\pi^{-1/2}y^{-3/2}dydx$. Therefore, by \eqref{eq:dispabove2}, $(\psi^+_{\pi t})_{t\geq0 }$ is distributed as $(\psi_t^{\IIC})_{t\geq0}$
and, consequently 
\[(\epsilon X^{\IIC}_{\epsilon^{-3}t})_{t\geq0}\to(Z^{\IIC}_t)_{t\geq0}\]
in distribution in the $M_1$ topology.

It only remains to strengthen the convergence to the uniform topology. It is a known fact that convergence in the $M_1$ topology coincides with convergence in the uniform topology when the limiting function is continuous. On the other hand it can be shown that $\psi^{\IIC}$ is continuous and therefore, $(B^+_{\psi^{\IIC}_t})_{t\geq0}$ is also continuous. This, however, does not immediately implies that $(\epsilon^{1/3} X^{\IIC}_{\epsilon^{-1}t})_{t\geq0}$ converges to $(B^+_{\psi^{\IIC}_t})_{t\geq0}$ in the uniform topology. We also have to check measurability of the pre-images $\{\omega\in\Omega:\epsilon^{1/3} X^{\IIC}_{\epsilon^{-1}t}\in A\}$ for all $A$ open in $D(\R_+)$ with the uniform topology.

Since the times of jumps of $X^{\IIC}$ are contained in $\N$, we have that the range $\epsilon^{1/3} X^{\IIC}_{\epsilon^{-1}t}(\Omega)$ is separable in $(D(\R_+),U)$, where $U$ denotes the uniform topology. Hence, for each $A$ open in the uniform topology, there exist countable many $U$-balls (balls in the uniform metric) $(A_i)_{i\in\N}$ such that
\[A\cap \epsilon^{1/3} X^{\IIC}_{\epsilon^{-1}t}(\Omega)=\bigcup_{i\in\N} (\epsilon^{1/3} X^{\IIC}_{\epsilon^{-1}t}(\Omega)\cap A_i).\] Then $\{\omega\in\Omega:\epsilon^{1/3} X^{\IIC}_{\epsilon^{-1}t}\in A\}=\bigcup_{i\in\N} \{\omega\in\Omega:\epsilon^{1/3} X^{\IIC}_{\epsilon^{-1}t}\in A_i\}$. But each set $\{\omega\in\Omega :\epsilon^{1/3} X^{\IIC}_{\epsilon^{-1}t}\in A_i\}$ is measurable because they are pre-images of $U$-balls and the $U$-balls can be written as countable intersections of finite-dimensional sets.

 Now we define probability measures $P_\epsilon, P$ in $(D(0,T),U)$. Let $A$ be a $U$-open set. Let $P_\epsilon[A]=\Pb[\epsilon^{1/3} X^{\IIC}_{\epsilon^{-1}t}\in A]$ and $P[A]=\Pb[Z^{\IIC}_t\in A]$ for each $A$.
 Let $C$ denotes the set of continuous functions. Let $A$ be a $U$-open, $A_U$ be its $U$ closure and $A_{M_1}$ be its $M_1$ closure. Then $\limsup_\epsilon P_\epsilon[A]\leq \limsup_\epsilon P_\epsilon[A_{M_1}]\leq P[A_{M_1}]=P[A_{M_1}\cap C]$, where the second inequality follows from the fact that $P_\epsilon\to P$ in the $M_1$ topology and the last equality follows from $P[C]=1$. But $(A_{M_1}\cap C)\subset A_U$ because, as we have said, the $M_1$ topology coincides with convergence in the uniform topology when the limiting function is continuous (here we are implicitly using the fact that the $M_1$ topology is metrizable, see \cite[Theorem 12.5.1]{whi02}). Therefore $\limsup_\epsilon P_\epsilon[A]\leq P[A_U]$ which implies that $P_\epsilon\to P$ in $(D(0,T),U)$. 

\section{Convergence of local times for the IIC}\label{sect_local_time}

In this section we prepare the proof of Assumption L for $X^\IIC$ and Assumption $\tilde{\text{L}}$ for $X^\IPC$. During the exposition, we will focus on $X^\IIC$, nevertheless, as we will see in Section \ref{s:proofoftheorem2}, the same results can be applied for $X^\IPC$.

Let $\B_1$ be a finite random tree having the distribution of one of the branches of the IIC. That is, $\B_1$ is the connected component of the root under critical percolation on a tree in which the root has degree $1$ and every other vertex has degree $3$. 
Note that, for the case of $X^\IIC$, $\pi^{d(\epsilon)}$ in Assumption L is the distribution of $\tilde{\sigma}[\B_1]$ conditioned on $\E[\tilde{\sigma}[\B_1]]=d(\epsilon)$, where $\tilde{\sigma}[\B_1]$ is as in \eqref{eq:sigmatilde}. On the other hand, as proved in \cite[Lemma 2.28]{Kesten}, for any rooted tree $I$, \[\E[\tilde{\sigma}[I]]=|I|.\] Therefore, for $X^{\IIC}$, $\pi^{d(\epsilon)}$ equals $\tilde{\nu}[\B^{d(\epsilon)}]$, where $\B^n$ denotes a random tree having the law of $\B_1$ conditioned on $|\B_1|=n$. 

 Assumption L states the convergence in distribution of the random Laplace exponent $\Psi_\epsilon(\tilde{\nu}[\B^{d(\epsilon)}])$.
 Next, we will show how $\Psi_\epsilon(\tilde{\nu}[\B^{d(\epsilon)}])$ is related to the inverse local time at the root of the simple random walk on $\B^{d(\epsilon)}$. Let $I$ be a rooted tree and $(Y[I]_k)_{k\in\N_0}$ be discrete-time, symmetric random walk on $I$ started at the root $\rho$. The local time at the root is
\be\label{eq:definitionofdiscretelocaltime}
l[I]_{t}:=\sum_{i=0}^{\lfloor t \rfloor} 1_{\{Y[I]_i=\rho\}},
\ee
and the inverse local time is
 \[l^{\scriptscriptstyle-1}[I]_t:=\min\{s\geq0:l[I]_s>t\}.\]

Note that $l^{\scriptscriptstyle-1}[I]_k$ is the sum of the duration of the $k$ first excursions of $Y[I]$ away from the root. This is a discrete time process and  by randomizing the times of jumps of $l^{\scriptscriptstyle-1}[I]$ (making waiting times exponential of parameter one instead of constant equal to one), we get a compound Poisson process of intensity $1$ whose jumps are distributed according to $\nu[I]$.
 On the other hand, is a standard fact that, for each $\nu\in M_1(\R_+)$, $\Psi_\epsilon(\nu)$ in \eqref{eq:psiepsilon} is the Laplace exponent of a compound Poisson process of rate $\epsilon^{-1}$ and size jump distribution $\nu(\cdot)$, scaled by a factor $q(\epsilon)$.
   Therefore, since the intensity of the compound Poisson process converges to infinity as $\epsilon\to0$, if the compound Poisson subordinator with Laplace exponent $\Psi_\epsilon(\nu[I])$ converges in distribution to some process, then the rescaled inverse local time $q(\epsilon)l^{\scriptscriptstyle-1}[I]_{\epsilon^{\scriptscriptstyle-1}k}$ converges to the same limit.  Therefore, Assumption L is, as we will see, equivalent to the convergence of the process $(q(\epsilon) l^{-1}[\B^{d(\epsilon)}]_{\epsilon^{-1}k})_{k\in\N}$.
  The main result of this section is Proposition \ref{prop:annealedlocal}, in which we prove that the rescaled local times
   \[\left(\epsilon l[\B^{d(\epsilon)}]_{q(\epsilon)^{\scriptscriptstyle-1}k}\right)_{k\in\N_0}\]
    converge, as $\epsilon\to0$, to the local time at the root of the Brownian motion in the Continuum Random Tree. This will imply that
   \[\left(q(\epsilon)l^{\scriptscriptstyle-1}[\B^{d(\epsilon)}]_{\epsilon^{\scriptscriptstyle-1}k}\right)_{k\in\N_0}\]
   converges to the inverse local time at the root of the Brownian motion on the CRT as $\epsilon\to0$. From that, it will follow that $\Psi_\epsilon(\nu[\B^{d(\epsilon)}])$ converges in distribution to the random Laplace exponent of the inverse local time at the root of the Brownian motion on the CRT, only multiplied by a constant factor. Finally, we will show that $\Psi_\epsilon(\tilde{\nu}[\B^{d(\epsilon)}])$ converges to the inverse local time at the root of the Brownian motion on the CRT. This will prove Assumption L for $X^\IIC$.

\subsection{Preliminaries}\label{sect_def_crt}
 This subsection is devoted to recall some known facts about discrete and continuous random trees and processes taking values on them. Those facts will be used to state and prove the main result of this section.
\subsubsection{Random trees}\label{s:crt}
We start by describing the \emph{search-depth} process which is a well-known way of representing trees through excursions.
Let $\mathrm{T}$ be an ordered, rooted tree having $n$ vertices.
  Let $\tilde{w}:\{1,2,\dots,2n-1\}\to \mathrm{T}$ be defined as follows. Set
$\tilde{w}(1)=\textrm{root of }\mathrm{T}$. Given $\tilde{w}(i)$, set $\tilde{w}(i+1)$ as the first (in the order of $\mathrm{T}$) descendant of $\tilde{w}(i)$ which is not on $\{\tilde{w}(k):k=1,\dots,i\}$. If all the descendants of $\tilde{w}(i)$ are in $\{\tilde{w}(k):i=1,\dots,i\}$, then set $\tilde{w}(i+1)$ as the progenitor of $\tilde{w}(i)$. The function $\tilde{w}$ is called the \textit{depth-first search around $\mathrm{T}$}.
In other words, suppose $\mathrm{T}$ is embedded in the plane in such a way that children are ``above" their progenitor and siblings are ordered from left to right according to their order on $\mathrm{T}$. Then $\tilde{w}$ moves along the vertices of $\mathrm{T}$ ``clockwise" (according to the embedding in the plane), starting from the root and ending on the root.

  Define the \textit{search-depth process} $\omega:[0,1]\to [0,\infty)$ by
\begin{equation}\label{d:search-depth}
\omega(i/2n):=d_{\mathrm{T}}(\textrm{root},\tilde{\omega}(i)),\hspace{1cm}1\leq i\leq2n-1
\end{equation}
where $d_{\mathrm{T}}$ is the graph distance on $\mathrm{T}$.
We also set $\omega(0)=\omega(1)=0$ and extend $\omega$ to the whole interval $[0,1]$ by linear interpolation.

It is not hard to see that one can reconstruct a tree from its search depth process. This idea has been exploited by Aldous in \cite{AldousCRT3} to construct ``continuous trees" starting from continuous excursion. We proceed to recall that procedure.
Let
\[\W:=\{w:[0,1]\to[0,\infty): w \textrm{ is continuous}; w(t)>0\textrm{ if and only if } t\in(0,1)\}\]
 be the space of (positive) excursions away from $0$ of duration $1$.
Given $w\in \W$, we define a pseudometric $d_w$ over $[0,1]$ by
\[
d_w(s,t):=w(s)+w(t)-2\inf\{w(r):r\in[s\wedge t, s\vee t]\}.
\]
Define the equivalence relation $\sim$ on $[0,1]$ by stating that $s\sim t$ if and only if $d_w(s,t)=0$. Then define the topological space $\T_w:=[0,1] \slash \sim$.

 We denote by $[r]$ the equivalence class of $r\in[0,1]$. We can endow $\T_w$ with a metric $d_{\T_w}([s],[t]):=d_w(s,t)$.
The space $\T_w$ is arc-connected and contains no subspace homeomorphic to the circle.
Moreover $d_{\T_w}$ is a \textit{shortest-path} metric, that is, $d_{\T_w}$ is additive along the non-self intersecting paths of $\T_w$. In other words, $\T_w$ is an $\R$\emph{-tree} (real tree).
The Lebesgue measure $\lambda$ on $[0,1]$ induces a probability measure $\mu_{\T_w}$ over $\T_w$ by
\[
\mu_{\T_w}(A):=\lambda(\{t\in[0,1]:[t]\in A\})
\]
for any Borelian $A\subset\T_w$.\\

Now, let $W=(W_t)_{t\in[0,1]}$ be a random process defined on a probability space $(\X,\mathcal{G},P)$ having the law of a normalized Brownian excursion. Clearly, $W$ can be viewed as a random object taking values in $\W$. Thus, starting from the Brownian excursion $W$, the previous procedure allows us to construct a random $\R$-tree denoted $\mathfrak{T}$ (or $\mathfrak{T}_W$ if we want to emphasize the role of the excursion), equipped with a shortest-path metric $d_{\mathfrak{T}}$ and a measure $\mu_{\mathfrak{T}}$.
\begin{definition}\label{definitioncrt}
 The triple $(\mathfrak{T},d_{\mathfrak{T}},\mu_{\mathfrak{T}})$ is the \emph{Continuum Random Tree} (CRT).
\end{definition}

Having defined the CRT, we turn our attention to the issue of convergence of rescaled discrete trees to the CRT.
 Aldous in \cite[Theorem 20]{AldousCRT3} showed that the convergence of a rescaled sequence of discrete, ordered, rooted trees to a continuum random tree (in a suitable topology) is equivalent to the convergence of their respective search-depth processes. Furthermore in \cite[Theorem 23]{AldousCRT3} it is shown that the critical Galton-Watson trees conditioned on having $n$ vertices scales to the CRT as $n\to \infty$ (up to an unimportant factor 2).

We finish our review of random trees with some definitions that will be used later.
Let $\K$ be an $\R$-tree and $A$ be a subset of $\K$. We will suppose that $\K$ has a distinguished point $\rho$ which we will regard as the root. We define the subspace $r(\K,A)$ as
\be
r(\K,A):=\bigcup_{x\in A}[[\rho,x]],
\ee
where $[[\rho,x]]$ denotes the unique non-self intersecting path between $\rho$ and $x$.
This subspace is clearly an $\R$-tree. Moreover, if $A$ is finite, $r(\K,A)$ is closed and is called the {\bf reduced sub-tree}.

Given a tree $\mathrm{T}$ (continuous or discrete), and a sub-tree of it $\mathrm{T}^\prime$, we can define the projection $\phi_{\mathrm{T},\mathrm{T}^{\prime}}$ of $\mathrm{T}$ onto $\mathrm{T}^{\prime}$ by simply stating that, for each $x\in \mathrm{T}$, $\phi_{\mathrm{T},\mathrm{T}^{\prime}}(x)$ is the point on $\mathrm{T}^\prime$ which is closest to $x$. The uniqueness of the projection follows easily from the tree structure of $\mathrm{T}$ and $\mathrm{T}^\prime$.

\subsubsection{Brownian Motion on the Continuum Random Tree}\label{s:BMCRT}
Next we will recall the definition of Brownian motion taking values in $\R$-trees, in particular, the Brownian motion in the Continuum Random Tree.

 Let $\K$ be a locally compact $\R$-tree equipped with a shortest-path metric $d_{\K}$ and a Radon measure $\nu$. We will assume that $\nu(A)>0$ for any non-empty open set $A\subset\K$.
  Fox all $x,y,z\in\K$ we define the branching point between $x,y$ and $z$ as the unique point $b^{\K}(x,y,z)$ that satisfies
\be
\{b^{\K}(x,y,z)\}:=[[x,y]]\cap[[x,z]]\cap[[y,z]],
\ee
where $[[x,y]]$ denotes the unique non-self intersecting path between $x$ and $y$.

 Let $((B^{\K,\nu}_t)_{t\geq0},(P^{\K,\nu}_z)_{z\in\K})$ be a reversible Markov process taking values in $\K$ with the following properties
\begin{enumerate}
\item Continuous sample paths.
\item Strong Feller property.
\item Reversible with respect to its invariant measure $\nu$.
\item For $x,y\in\K,x\neq y$ we have
\[P^{\K,\nu}_z(\sigma_x<\sigma_y)=\frac{d_{\K}(b^{\K}(z,x,y),y)}{d_{\K}(x,y)},\textrm{ for all } z\in\K \]
where $\sigma_x:=\inf\{t>0:B_t^{\K,\nu}=x\}$ is the hitting time of $x$.
\item For $x,y\in\K$, the mean occupation measure for the process started at $x$ and killed on hitting $y$ has a density w.r.t.~to $\nu$ given by
\[2d_{\K}(b^{\K}(z,x,y),y)\nu(dz)\textrm{ for all } z\in\K.\]
\end{enumerate}
In section 5.2 of \cite{AldousCRT2}, Aldous claims that one can prove that such a process must be unique (in law). This allows to make the following definition.
\begin{definition}\label{definitionofbmonadendrite}
The process $B^{\K,\nu}$ is the Brownian motion on $(\K,d_{\K},\nu)$.
\end{definition}
 The existence of a process satisfying the definitions above, in the case where $\K$ is a locally compact $\R$-tree was first given by Krebs in \cite{krebs1995brownian}. Also Proposition 2.2 in \cite{rwrt} uses results from Kigami \cite{kigami1995harmonic} to define the $B^\K$ in a more concise way.  \begin{definition}
We define the Brownian motion on the Continuum Random Tree as
$B^{\mathfrak{T}_W,\mu_{\mathfrak{T}_W}}$, where $W$ is distributed as a normalized Brownian excursion. For simplicity, for each realization $w$ of $W$, $B^{\mathfrak{T}_w,\mu_{\mathfrak{T}_w}}$ will sometimes be denoted $B^w$.
\end{definition}

Lemma 2.5 in \cite{rwrt} ensures that for $P$-a.e. $w\in\W$ there exist jointly-continuous local times $(L_w(x,t))_{x\in\mathfrak{T}_w,t\geq0}$ for $B^{\mathfrak{T}_w,\mu_{\mathfrak{T}_w}}$.
The inverse local time at $x\in \mathfrak{T}_w$ is defined as
\[L^{-1}_w(x,t):=\inf\{s:L_w(x,s)>t\}.\]
 We finish this review on the Brownian motion on the CRT by defining the \emph{annealed laws} of the local time (and its inverse) at the root of the Brownian motion on the Continuum random tree, $\mathbb{G},\mathbb{G}^*\in M_1(D[0,\infty)),$ as
 \be\label{eq:annealedlocaltime}
\mathbb{G}(A):=\int_{\W}P^{\mathfrak{T}_w,\mu_{\mathfrak{T}_w}}(L_w(\rho,\cdot)\in A)P(dw)
\ee
 and
\be\label{annealedlocaltime}
\mathbb{G}^*(A):=\int_{\W}P^{\mathfrak{T}_w,\mu_{\mathfrak{T}_w}}(L^{-1}_w(\rho,\cdot)\in A)P(dw)
\ee
for all $A$ Borelian of $D[0,\infty)$ with the Skorohod topology. In what follows, sometimes we will consider the law $\mathbb{G}(A)$ as defined in the uniform topology instead of the Skorohod topology.
 
 To finish this subsection, we state and prove a preliminary result concerning local times, which was used in the proof of  Lemma~\ref{assumptionL}.
\begin{lemma}\label{l:irrelevanceoftheconstant}
 Let $w$ be a fixed realization of the normalized Brownian excursion $W$. Then, for each $c>0$ $L_w(\rho,\cdot)$ (the local time at the root of the Brownian motion on $\mathfrak{T}_w$) is distributed as $L_{cw}(\rho,\cdot)$ (the local time at the root of the Brownian motion on $\mathfrak{T}_{cw}$)
\end{lemma}
\begin{proof}
 Let $B^{\mathfrak{T}_{w},\mu_{\mathfrak{T}_w}}$ be the Brownian motion on $(\mathfrak{T}_{w},d_{\mathfrak{T}_w},\mu_{\mathfrak{T}_w})$. We recall that $\mathfrak{T}_{w}$ is defined as $[0,1]/\sim$ where $x\sim y$ iff $d_{w}(x,y)=0$. Therefore $\mathfrak{T}_{w}=\mathfrak{T}_{cw}$ as sets. Hence $B^{\mathfrak{T}_{w},\mu_{\mathfrak{T}_w}}$ can be regarded as a process on $\mathfrak{T}_{cw}$. It can be checked that $B^{\mathfrak{T}_{w},\mu_{\mathfrak{T}_w}}$ is the Brownian motion on $(\mathfrak{T}_{cw},d_{\mathfrak{T}_{cw}},\mu_{\mathfrak{T}_{cw}})$ according to Definition \ref{definitionofbmonadendrite}. The result follows.
\end{proof}

\subsection{Statement and proof of the convergence of local times}\label{sect_local_time1}
Recall that $\B^n$ denotes the law of a branch of the IIC conditioned on having $n$ vertices.
In this subsection we will prove that the rescaled local time of the random walk on $\mathcal{B}^n$ converge to the local time of the Brownian motion in the continuum random tree. During the argument we will make use of many ideas of \cite{rwrt}.

Let us call $(\Omega,\mathcal{F},P)$ the probability space over which the conditioned branches $\mathcal{B}^n$ are defined. Recall that, for any realization $\mathcal{B}^n(\omega)$ of the random trees, $(l[\mathcal{B}^n(\omega)]_t)_{t\geq0}$ denotes the local time process of a simple random walk on $\mathcal{B}^n(\omega)$. For each $\omega\in\Omega$, let $P_n^\omega$ denote the law of the process $l[\mathcal{B}^n(\omega)]$.
Let us defined the \emph{annealed law} of the local time as \[\mathbb{P}_n(A):=\int_{\Omega}P_n^\omega[ A ] \mathbb{P}(d\omega),\]
for all $A$ Borelian of $D[0,\infty)$ under the topology of uniform convergence over compact subsets of $[0,\infty)$. The main result of this subsection is:
\begin{proposition}\label{prop:annealedlocal}
 Let $(L(\rho,t))_{t\geq0}$ be a process having the annealed law $\mathbb{Q}$ of the local time of the BM on the CRT (see \eqref{eq:annealedlocaltime}). Then
\[(n^{-1/2}l[\B^n]_{tn^{3/2}})_{t\geq0}\stackrel{n\to\infty}{\to}\left(\frac{1}{2}L(\rho,t)\right)_{t\geq0},\]
in $\mathbb{P}_n$-distribution in $D[0,\infty)$ endowed with the topology of uniform convergence over compact subsets of $[0,\infty)$.
\end{proposition}
The proof of the proposition above is obtained trough a coupling between the random trees $\B_n$ and the CRT. We now pass to describe the coupling.


Let $w^n$ denote respectively the search-depth processes of $\B^n$. As in the argument leading to \eqref{eq:convergenceofsearchdepthprocesses} above, we can use \cite[Theorem 23]{AldousCRT3} to deduce that
 \begin{equation}\label{eq:preassumption1}
  \left(n^{-1/2}w^n(t)\right)_{t\in[0,1]}\stackrel{d}{\Rightarrow}\left(\sqrt{2}W_t\right)_{t\geq0}\quad \text{as } n\to\infty
 \end{equation}
 on $C[0,1]$ endowed with the uniform topology, where $(W_t)_{t\in[0,1]}$ is the normalized Brownian excursion. By virtue  of the Skorohod representation theorem, we can find coupled processes $\bar{w}^n,n\in\N$  and $\bar{W}$ defined on a common probability space $(\bar{\Omega},\bar{\mathcal{F}},\mathbb{Q})$ such that $\bar{w}^n$ is distributed as $w^n$, $\bar{W}$ is distributed as $W$ and
 \begin{equation}\label{eq:convergenceofcoupledsearchdepthprocesses}
 \left(n^{-1/2}\bar{w}^n_t\right)_{t\in[0,1]} \stackrel{u}{\to} \left(\sqrt{2}\bar{W}_t\right)_{t\in[0,1]},\quad \mathbb{Q}\text{-a.s.,}
 \end{equation}
 where $\stackrel{u}{\to}$ denotes uniform convergence.
Note that, since an ordered tree can be reconstructed from its search depth process, we can consider the random trees $(\bar{\B}^n)_{n\in\N}$ corresponding to the search-depths $\bar{w}_n$. The sequence $(\bar{\B}^n)_{n\in\N}$ has the same distribution as that of the original sequence $(\B^n)_{n\in\N}$.

Let $\mathrm{T}_n$ be a fixed realization of $\bar{B}^n$. Recall from \eqref{eq:definitionofdiscretelocaltime} that $l[\mathrm{T}_n]$ denotes the local time process at the root of a random walk on $\mathrm{T}_n$. Recall also that, although $\mathrm{T}_n$ is a deterministic tree, $l[\mathrm{T}_n]$ is a random process. The next proposition  is at the heart of the proof of Proposition \ref{prop:annealedlocal} states the \emph{quenched} convergence of the local times.

 
 \begin{proposition}\label{convergenceoflocaltimes}
 For $\mathbb{Q}$-a.e. realization $(\mathrm{T}_n)_{n\in\N}$ of the random trees $(\bar{\B}^n)_{n\in\N}$ and $\mathbb{Q}$-a.e. realization $w$ of the Brownian excursion $(\bar{W}_t)_{t\in[0,1]}$, it holds that 
 \[(n^{-1/2}l[\mathrm{T}_n]_{tn^{3/2}})_{t\geq0}\stackrel{n\to\infty}{\to}\left(\frac{1}{2}L_{\sqrt{2}w}(\rho,t)\right)_{t\geq0}\] in distribution in $D[0,\infty)$ endowed with the topology of uniform convergence over compact subsets of $[0,\infty)$.
\end{proposition}

As a direct corollary of Proposition \ref{convergenceoflocaltimes} we have the convergence of the corresponding inverse local times. 

\begin{corollary}\label{quenchedconverenceofinverselocaltimes}
For $\mathbb{Q}$-a.e. realization $(\mathrm{T}_n)_{n\in\N}$ of the random trees $(\bar{\B}^n)_{n\in\N}$ and $\mathbb{Q}$-a.e. realization $w$ of the Brownian excursion $(\bar{W}_t)_{t\in[0,1]}$, it holds that 
\[(n^{-3/2}l^{-1}[\mathrm{T}_n](n^{1/2}t))_{t\geq0}\stackrel{n\to\infty}{\to}(L^{-1}_{\sqrt{2}w}(2t))_{t\geq 0}\]
in $D[0,\infty)$ endowed with the Skorohod $M_1$ topology. \end{corollary}
 For the definition of the Skorohod $M_1$ topology we refer to \cite[\S3.3]{whi02}.
 \begin{proof}[Proof of Corollary \ref{quenchedconverenceofinverselocaltimes}]
 By \cite[Lemma 13.6.3]{whi02} we know that the inversion map on $(D(\mathbb R_+), M_1)$ is continuous. Hence we can obtain convergence in the Skorohod $M_1$ topology for the corresponding inverted processes.
\end{proof}

Before going to the proof of Proposition \ref{convergenceoflocaltimes}, we will show that it implies Proposition \ref{prop:annealedlocal}.
\begin{proof}[Proof of Proposition \ref{prop:annealedlocal}]
By definition of the annealed law, and the coupling of Proposition \ref{convergenceoflocaltimes}, we have
\begin{equation}
\begin{aligned}
&\mathbb{P}_n\left[(n^{-1/2}l[\B^n]_{tn^{3/2}})_{t\geq0}\in A\right]:=\\
:=&\int_\Omega P_\omega[(n^{-1/2}l[\B^n(\omega)]_{tn^{3/2}})_{t\geq0}\in A]\mathbb{P}(d\omega)\\
=&\int_{\bar{\Omega}} P_\omega[(n^{-1/2}l[\B^n(\omega)]_{tn^{3/2}})_{t\geq0}\in A]\mathbb{Q}(d\omega).
\end{aligned}
\end{equation}
Using Proposition \ref{convergenceoflocaltimes} (and the bounded convergence theorem), we get that
\[\int_{\bar{\Omega}} P_\omega[(n^{-1/2}l[\B^n(\omega)]_{tn^{3/2}})_{t\geq0}\in A]\mathbb{Q}(d\omega)\to\int_{\bar{\Omega}} P_\omega\left[\left(\frac{1}{2}L_{\sqrt{2}w}(\rho,t)\right)_{t\geq0}\in A\right]\mathbb{Q}(d\omega)
\]
for every $A$, continuity set of $\mathbb{G}$. This shows the convergence in distribution towards $\frac{1}{2}L_{\sqrt{2}w}(\rho,\cdot)$. We can get rid of the $\sqrt{2}$ factor by Lemma \ref{l:irrelevanceoftheconstant}. This finishes the proof.
\end{proof}

The rest of the subsection deals with the proof of
 Proposition \ref{convergenceoflocaltimes}. We will need to consider subtrees of $\mathfrak{T}_w$. In order to span subtrees we use an i.i.d.~sequence of random variables $U=(U_i)_{i\in\N}$, uniformly distributed on $[0,1]$, which are defined in the same probability space $(\bar{\Omega},\bar{\mathcal F},\mathbb{Q})$ as the Brownian excursion $\bar{W}$ and are independent of it.



 Let $(w,(u_i)_{i\in\N}$) be a realization of the pair $(\bar{W},(U_i)_{i\in\N})$. We will span subtrees of $\mathfrak{T}_w$ using $(u_i)_{i\in\N}$.
For $k\in\N$, we define the reduced sub-tree $\mathfrak{T}_w(k)$ as
\begin{equation}\label{eq:defofreduced}
\mathfrak{T}_w(k):=r(\mathfrak{T}_w,\{[u_i]:i\leq k\}).
\end{equation}
 Since $\mathfrak{T}_w(k)$ is composed of a finite number of line segments, we can define the Lebesgue measure $\lambda^{(k)}_{w,u}$ over $\mathfrak{T}_w$. Moreover we will assume that $\lambda^{(k)}_{w,u}$ is normalized to become a probability measure.

Similarly, for $n$ fixed, we would like to use $(u_i)_{i\in\N}$ to span subtrees of a fixed realization $\mathrm{T}_n$ of the random trees $\bar{\B}_n$.
Let
\begin{equation}\label{eq:defofgamma}
\gamma_n(t):=\left\{\begin{array}{ll}

\lfloor2nt\rfloor/2n & \textrm{ if } w^n(\lfloor2nt\rfloor/2n)\geq w^n(\lceil2nt\rceil/2n),\\
\lceil 2nt\rceil/2n & \textrm{ otherwise, }
             \end{array}\right.
\end{equation}
where $w^n$ is the search-depth of $\mathrm{T}_n$.
This function is constructed so that, if $U$ is uniformly distributed over $[0,1]$, then $\tilde{w}_n(\gamma_n(U))$ is uniformly distributed over the vertices of $\mathrm{T}_n$, where $\tilde{w}_n$ is the depth-first search around $\mathrm{T}_n$. We define the discrete reduced sub-trees as
\[
\mathrm{T}_n(k):=r(\mathrm{T}_n,\{\tilde{w}_n(\gamma_n(u_i));i\leq k\}).
\]


\begin{remark}
In what follows, we are going to cite several results from \cite{rwrt} which work under the hypothesis that the pair $((w^n)_{n\in\N},(u_i)_{i\in\N})$ satisfy the so called \emph{Assumption 1}. More precisely, the setting in \cite{rwrt} is to consider, for each $n\in\N$, a sequence $(u^{(n)}_i)_{i\in\N}$ to span points from $\mathrm{T}_n$ (instead of a single sequence $(u_i)_{i\in\N}$ as we do).
Assumption 1 in \cite{rwrt} asks that, for each $n$, the sequence $(u_i^{(n)})_{i\in\N}$ is dense in $[0,1]$, and also
\begin{equation}\label{eq:ass1}
(n^{-1/2}w_n,u^{(n)})\to(w,u)
\end{equation}
in $C([0,1],\R_+)\times[0,1]^\N$, for some $(w,u)\in\Gamma$, where $[0,1]^\N$ is endowed with the product topology and $\Gamma$ is a specific subset of  $C([0,1],\R_+)\times[0,1]^\N$ which satisfies \[\mathbb{Q}[(\bar{W},(U_i)_{i\in\N})\in\Gamma]=1.\]
As we have said, in this article we choose the sequence $u^{(n)}$ independent of $n$, that is $u_i^{(n)}=u_i$, for all $n,i\in\N$. Since $(u_i)_{i\in\N}$ is $\mathbb{Q}$-a.s. dense in $[0,1]$, it is easy to deduce from \eqref{eq:preassumption1} that $((w^n)_{n\in\N},(u_i)_{i\in\N})$ satisfy Assumption 1, $\mathbb{Q}$-almost surely\footnote{Note that, strictly speaking, starting form \eqref{eq:preassumption1} we can deduce the $\mathbb{Q}$-a.s. occurrence of \eqref{eq:ass1}, but only with a $\sqrt{2}$ factor in front of $w$. Nevertheless, all the results that we are going to cite from \cite{rwrt} still hold when we add the $\sqrt{2}$ factor in the definition of Assumption 1.}. Hence when referring to the results of \cite{rwrt} that work under Assumption 1, we can (and will) instead state that they hold $\mathbb{Q}$-almost surely, without making direct reference to Assumption 1.
\end{remark}

 Define
\[
A^{(k)}_t:=\int_{\mathfrak{T}_w(k)}L_w(x,t)\lambda^{(k)}_{w,u}(dx),
\]
where we recall that $L_w$ is the local time of the Brownian motion $B^w$ on $\mathfrak{T}_w$.
Let also
\be\label{definitionoftau}
\tau^{(k)}(t):=\inf\{s:A_s^{(k)}>t\}.
\ee
and
\be\label{eq:definitionofBk}
B^{(k)}_t:=B^w_{\tau^{(k)}(t)}.
\ee
Lemma 2.6 in \cite{rwrt} ensures that the process $B^{(k)}$ is the Brownian motion on $(\mathfrak{T}_w(k),\lambda^{(k)}_{w,u})$ (according to Definition \ref{definitionofbmonadendrite}).
Moreover Lemma 3.3 in \cite{rwrt} implies that, $\mathbb{Q}$-almost surely, $B^{(k)}$ has jointly continuous local times $(L_w^{(k)}(x,t),t\geq0,x\in\mathfrak{T}_w(k))$.

Next we will state a lemma which will be used in the proof of Proposition \ref{convergenceoflocaltimes}.
Let $\Lambda_n^{(k)}:=n^{-1/2}\#\{\text{vertices of }\mathrm{T}_n(k)\}$.
\begin{lemma} \label{l:convergenceoflocaltimes:discretetocontinuous:reduced}
For $\mathbb{Q}$-a.e. realization $(\mathrm{T}_n)_{n\in\N}$ of the random trees $(\bar{\B}_n)_{n\in\N}$ and $\mathbb{Q}$-almost every realization $((w(t))_{t\in[0,1]},(u_i)_{i\in \N})$ of the pair $((\bar{W}_t)_{t\in[0,1]},(U_i)_{i\in\N})$, it holds that 
\[
(n^{-1/2}l[\mathrm{T}_n(k)]_{tn\Lambda_n^{(k)}})_{t\geq0}\stackrel{d}{\to}\left(\frac{1}{2}L^{(k)}_{\sqrt{2}w}(\rho,t)\right)_{t\geq0},
\]
with the topology of uniform convergence over compact intervals of time.
\end{lemma}

\begin{proof}[Proof of Lemma \ref{l:convergenceoflocaltimes:discretetocontinuous:reduced}]
The discrete tree $\mathrm{T}_n(k)$ can be regarded as an $\R$-tree by adding line segments between adjacent vertices. Specifically, since we are interested in rescalings of $\mathrm{T}_n(k)$, we will consider $\mathrm{T}_n(k)$ as an $\R$-tree with a shortest path metric $d_{\mathrm{T}_n(k)}$ by adding line segments of length $n^{-1/2}$ between each pair of adjacent vertices.

 Let $\lambda_{n,k}$ be the Lebesgue measure on $\mathrm{T}_n(k)$ with respect to the metric $d_{\mathrm{T}_n(k)}$, so that $\lambda_{n,k}(\mathrm{T}_n(k))=\Lambda_n^{(k)}$. Let $(B^{n,k}_t)_{t\geq 0}$ be the Brownian motion on $(\mathrm{T}_n(k),d_{\mathrm{T}_n(k)},\lambda_{n,k})$ and $(\bar{L}^{n,k}(x,t))_{x\in \mathrm{T}_n(k),t\geq0}$ be a jointly continuous version of its local time (whose $\mathbb{Q}$-a.s.~existence in guaranteed by Lemma 2.5 in \cite{rwrt}).

 Let $(v^k_i)_{i\leq l^k}$ be the set composed of the root, the leaves and branching points of $\mathfrak{T}_{\sqrt{2}w}(k)$. Also let $(e_i)_{i\leq l^{k}-1}$ denote the line segments of $\mathfrak{T}_{\sqrt{2}w}(k)$ which join the points $(v^k)_{i\leq l^k}$.
 Lemma 4.1 in \cite{rwrt} states that, $\mathbb{Q}$-almost surely, for each $k$ fixed, $\mathrm{T}_n(k)$, regarded as an $\R$-tree, converges to $\mathfrak{T}_{\sqrt{2}w}(k)$ as $n\to\infty$. Therefore, for $n$ large enough, $\mathrm{T}_n(k)$ is homeomorphic to $\mathfrak{T}_w(k)$.  Moreover, we can define the homeomorphism $\Upsilon_{n}^{(k)}:\mathfrak{T}_{\sqrt{2}w}(k)\mapsto \mathrm{T}_n(k)$ which preserves order and is linear along the line segments $(e_i)_{i\leq l^k-1}$.

Let us define the distance $\bar{d}_{n,k}$ on $\mathfrak{T}_w(k)$ by
\[\bar{d}_{n,k}(x,y)=d_{\mathrm{T}_n(k)}(\Upsilon_{n}^{(k)}(x),\Upsilon_{n}^{(k)}(y)).\]
Let $\bar{\lambda}_{n,k}$ be the Lebesgue measure of $\mathrm{T}_n(k)$ with respect to the distance $\bar{d}_{\mathrm{T}_n(k)}$.
Let $\Upsilon^{(k)\leftarrow}_{n}$ denote the inverse of $\Upsilon^{(k)}_n$. By verifying the properties in Definition \ref{definitionofbmonadendrite} it can be shown that $(\Upsilon^{(k)\leftarrow}_{n}(B^{n,k}_t))_{t\geq0}$ has the law of the Brownian motion in $(\mathfrak{T}_{\sqrt{2}w}(k),\bar{d}_{n,k},\bar{\lambda}_{n,k})$.

By the $\mathbb{Q}$-a.s.~convergence of $\mathrm{T}_n(k)$ towards $\mathfrak{T}_{\sqrt{2}w}(k)$ (guaranteed by Lemma 4.1 in \cite{rwrt}), one can choose a family of constants $(\delta_n)_{n\in\N}\subset(0,1]$, $\delta_n\to1$ as $n\to\infty$, that satisfies
\begin{equation}\label{eq:croydon2012}
\delta_n\bar{d}_{n,k}(x,y)\leq d_{\mathfrak{T}_{\sqrt{2}w}(k)}(x,y)\leq \delta_n^{-1}\bar{d}_{n,k}(x,y) \quad \forall x,y \in \mathfrak{T}_{\sqrt{2}w}(k).
\end{equation}
Proposition 3.1 in \cite{Croydon2012}, ensures that under \eqref{eq:croydon2012}, the corresponding local times converge. Therefore, since the local time at the root of 
$(\Upsilon^{-1}_{\mathfrak{T}_{\sqrt{2}w}(k),\mathrm{T}_n(k)}(B^{n,k}_t))_{t\geq0}$ coincides with that of $(B^{n,k}_t)_{t\geq0}$, we have that, almost surely,
\be\label{eq:proposition3.1ofcroydon}
(\bar{L}^{n,k}(\rho,\Lambda_{n}^{(k)} t))_{t\geq0}\stackrel{d}{\to}(L_{\sqrt{2}w}^{(k)}(\rho,t))_{t\geq0}
\ee
as $n\to\infty$ in $C[0,\infty)$ with the topology of uniform convergence on compact sets.

Consider the search depth process $h^n$ of $\mathrm{T}_n$ (see \eqref{d:search-depth}).
 Set $h^{n,k}(0):=0$ and
\[
h^{n,k}(m):=\inf\left\{t\geq h^n(m-1): B^{n,k}_t\in V(T_n(k))-\{B^{n,k}_{h^{n,k}(m-1)}\}\right\}.
\]
Define \[\bar{J}^{n,k}_m:=B^{n,k}_{h^{n,k}(m)}.\]
Observe that the process $\bar{J}^{n,k}$ is a simple random walk on the vertices of $\mathrm{T}_n(k)$.
From now on we will assume that the local time $l[\mathrm{T}_n(k)]$ is constructed with respect to $\bar{J}^{n,k}$.
Finally, the reasoning in the proof of Lemma 4.8 in \cite{rwrt} can be used to show that, $\mathbb{Q}$-almost every realization $((\mathrm{T}_n)_{n\in\N},(u_i)_{i\in\N})$ of the pair  $((\B_n)_{n\in\N},(U_i)_{i\in\N})$, 
\be\label{eq:lemma4.8ofcroydon}
\Pb\left[ \left.\sup_{t\leq M}\left|2n^{-1/2}l[\mathrm{T}_n(k)]_{tn}-\bar{L}^{n,k}(\rho,t)\right|>\varepsilon  \right\vert (\B_n)_{n\in\N}=(\mathrm{T}_n)_{n\in\N}, (U_i)_{i\in\N}=(u_i)_{i\in\N}\right]\to 0
\ee
as $n\to\infty$, for all $M\geq0$.

The lemma follows from displays \eqref{eq:proposition3.1ofcroydon} and $\eqref{eq:lemma4.8ofcroydon}$.
\end{proof}
\begin{proof}[Proof of Proposition \ref{convergenceoflocaltimes}]
The proof will rely on Lemma \ref{l:convergenceoflocaltimes:discretetocontinuous:reduced}. 

We will relate $l[\mathrm{T}_n(k)]$ and $l[\mathrm{T}_n]$ using the following coupling:
Let $X^n$ be a simple random walk on $\mathrm{T}_n$ started at the root. Define $A^{n,k}(0):=0$ and
 \[
 A^{n,k}(m):=\min\left\{j\geq A^{n,k}(m-1):X^{n}_j\in \mathrm{T}_n(k)-\{X^n_{A^{n,k}(m-1)}\}\right\}.
 \]
The process $(J^{n,k}_m)_{m\geq 0}$ defined as
\[
 J^{n,k}_m=X^{n}_{A^{n,k}(m)}
 \]
is a simple random walk on $\mathrm{T}_n(k)$.

During the proof we will assume that the local times $l[\mathrm{T}_n]$ and $l[\mathrm{T}_n(k)]$ are defined in terms of $X^n$ and $J^{n,k}$ respectively.
Since by assumption we have that $\deg_{\mathrm{T}_n}(\rho)=1$ for all $n\in\N$ (where $\deg_{\mathrm{T}_n}(x)$ denote the degree of $x$ in $\mathrm{T}_n$ ) , we have that each excursion away from the root of $X^n$ is also a excursion away from the root of $J^{n,k}$. Therefore
\be\label{fromreducedlocaltimestolocaltimes}
l[\mathrm{T}_n(k)]_{j}=l[\mathrm{T}_n]_{A^{n,k}(j)}.
\ee

On the other hand, Corollary 5.3 in \cite{rwrt} implies that for all $M\geq0$
\be\label{l:corollary5.3ofcroydon}
\lim_{k\to\infty}\limsup_{n\to\infty}\Pb\left[\sup_{t\leq M}\left|n^{-3/2}A^{n,k}(\Lambda_n^{(k)}nt)-t\right|>\varepsilon\right]=0.
\ee

By virtue of the Skorohod representation theorem, we can assume that the convergence in Lemma \ref{l:convergenceoflocaltimes:discretetocontinuous:reduced} is almost sure. In particular, we can assume that $L^{(k)}_{\sqrt{2}w}$ and $l[\mathrm{T}_n(k)]$ are defined in the same probability space and, for each $M\geq0$ they satisfy
\[
\lim_{n\to\infty}\Pb\left[\sup_{t\leq M}\left|n^{-1/2}l[\mathrm{T}_n(k)]_{tn\Lambda_n^{(k)}}-\frac{1}{2}L^{(k)}_{\sqrt{2}w}(\rho,t)\right|\geq\epsilon \right]= 0.
\]
Therefore, by \eqref{fromreducedlocaltimestolocaltimes}, and the uniform continuity of $L^{(k)}_{\sqrt{2}w}(\rho,t)$ in $[0,M]$
\be\label{**}
\lim_{n\to\infty}
\Pb\left[\sup_{t\leq M}\left|n^{-1/2}l[\mathrm{T}_n]_{A^{n,k}(tn\Lambda_n^{(k)})}-\frac{1}{2}L^{(k)}_{\sqrt{2}w}(\rho,t)\right|\geq\epsilon \right]=0.
\ee
Lemma 3.4 in \cite{rwrt} states that
\be\label{Lemma 3.4}
L^{(k)}_{\sqrt{2}w}(\rho,t)=L_{\sqrt{2}w}(\rho,\tau^{(k)}(t)),
\ee
where $\tau^{(k)}(t)$ is as in \eqref{definitionoftau}.
Moreover, by  \cite[Lemma 3.1]{rwrt}, for each $M\geq0$
\be\label{eq:probabilityofC}
\mathbb{P}\left[\sup_{t\leq M}\left|\tau^{(k)}(t)-t\right|>\varepsilon\right]\to0
\ee
as $k\to\infty$.
Therefore, using the uniform continuity of $L_{\sqrt{2}w}(\rho,t)$ in $[0,M]$ together with \eqref{Lemma 3.4} and \eqref{eq:probabilityofC} we get
\begin{equation}\label{***}
\Pb\left[\sup_{t\leq M}\left|L^{(k)}_{\sqrt{2}w}(\rho,t)-L_{\sqrt{2}w}(\rho,t) \right|\geq\epsilon \right]\to0
\end{equation}
as $k\to\infty$.
Finally, combining \eqref{l:corollary5.3ofcroydon}, \eqref{**} and \eqref{***} we get that
\[
\Pb\left[\sup_{t\leq M}\left|n^{-1/2}l[\mathrm{T}_n]_{tn^{3/2}}-\frac{1}{2}L_{\sqrt{2}w}(\rho,t)\right|\geq\epsilon \right]\to 0\]
as $n\to\infty$, which proves our claim.

\end{proof}

\section{Convergence result for the IPC: proof of Theorem~\ref{thm:IPC} and Theorem~\ref{t:descriptionofzipc}}\label{sect_ipc}
\subsection{Assumption $\widetilde{\text{HT}}$ for $Z^{\IPC}$}\label{sect_IPCHT}

We recall that $\Lv_k$ is the branch emerging from the $k$-th vertex of the backbone of the IPC and $\mu_{\IPC}$ is the random measure appearing in (\ref{ipclimitmeasure}). The main result of this section is
\begin{lemma}\label{annealedipc}
Let $V^{\IPC}_x:=\sum_{k=1}^{\lfloor x \rfloor} m(\tilde{\nu}[\Lv_k])$ and $I_x:=\mu_{\IPC}(0,x]$  . Then
\[
(\epsilon^2 V^{\IPC}_{\epsilon^{-1} x})_{x\geq0}\stackrel{\epsilon\to0}
    \to (I_x)_{x\geq0}
\]
in distribution on $(D[0,\infty),J_1)$.
\end{lemma}

To prepare the proof of Lemma \ref{annealedipc}, we first need to provide some known facts about the IPC. Let $P_j$ be the weight of the $j$-th vertex of the backbone of the IPC $\I^{\infty}$ and $M_k:=\sup\{P_j:j>k\}$. By \cite [Proposition 2.1]{ipc} we have that, conditioned on a fixed realization of $(P_k)_{k\in\N}$, the sequence of branches $(\Lv_k)_{k\in\N}$ is an independent sequence of trees where each $\Lv_k$ is distributed as a supercritical percolation cluster on $\mathbb{T}_2^\ast$ with parameter $M_k$, conditioned to stay finite.
The percolation parameter $M_k$ corresponding to the cluster attached at $k\in \mathbb{N}$ decreases to $p_c=1/2$ as $k$ goes to $\infty$.
In fact, it can be shown (see \cite[Proposition 3.3]{ipc}) that for any $\epsilon>0$
\begin{equation} \label{environment}
(k[2M_{\left\lceil kt \right\rceil}-1 ])_{t>\epsilon} \stackrel{k\to\infty}{\to} (E_t)_{t>\epsilon}
\end{equation}
in distribution on $(D[\epsilon,\infty),J_1)$ where $E_t$ is the lower envelope of a homogeneous Poisson point process as in display (\ref{lt}).

One can use duality of percolation to see that a supercritical cluster with parameter $p$ conditioned to stay finite is distributed as a subcritical cluster
 with dual parameter $\tilde{p}$ which satisfies (see \cite[Lemma 2.2]{ipc})
\[
 p-p_c \sim p_c-\tilde{p} \textrm{  as  } p\downarrow p_c
\]
where $\sim$ denotes asymptotic equivalence.
Hence, using (\ref{environment}) we can show that, for each $\epsilon>0$
\begin{equation}\label{environmentII}
(k[1-2\tilde{M}_{\left\lceil kt \right\rceil} ])_{t>\epsilon} \stackrel{k\to\infty}{\to} (E_t)_{t>\epsilon}
\end{equation}
in distribution on $(D[\epsilon,\infty),J_1)$

Using display (\ref{environmentII}) and the Skorohod representation theorem we can find, for each $\epsilon>0$, copies of $(k[1-2\tilde{M}_{\left\lceil kt \right\rceil} ])_{t>\epsilon},k\in\N$ and $(E_t)_{t>\epsilon}$ in which the convergence in (\ref{environmentII}) holds almost surely. It will be more convenient to have copies which do not depend on $\epsilon$ and in which the almost sure convergence holds when restricted to $(\epsilon,\infty)$, for each $\epsilon>0$.
In order to do that we first prove this simple lemma
\begin{lemma}\label{scalingforthepercolationparameter}
$(k^{-1}(1-2\tilde{M}_{\lceil tk \rceil})^{-1})_{t\geq0}$ converge in distribution to $(E_t^{-1})_{t\geq0}$ in the Skorohod $J_1$ topology as $k\to\infty$.
\end{lemma}
\begin{proof}
Using display (\ref{environmentII}), continuity of $x\mapsto x^{-1}$ on $(\epsilon,\infty)$ and the continuous mapping Theorem we obtain that, for each $\epsilon>0$, $(k^{-1}(1-2\tilde{M}_{\lceil tk\rceil})^{-1})_{t\geq\epsilon}$ converges  to $(E_t^{-1})_{t\geq\epsilon}$. From this we can deduce convergence of finite dimensional distributions and tightness away from $0$ (for tightness in the Skorohod $J_1$ topology, see e.g., Theorem 15.6 in \cite{Billingsley68}). To deal with the behavior near $0$ we use the fact that $k^{-1}(1-2\tilde{M}_{\lceil \epsilon k\rceil})^{-1}$ converges in distribution to $E_\epsilon^{-1}$ and $E_\epsilon^{-1}$ converges in distribution to $\delta_0$ as $\epsilon\to0$. Also, the processes involved are increasing and positive. This gives tightness near $0$ and convergence of marginals at $t=0$.
\end{proof}
Using the previous lemma and the Skorohod representation theorem we can find a family of processes $(\bar{M}^{k}_t)_{t\geq0}$, $k\in\N$, and a process $(\bar{E}_t)_{t\geq 0}$ defined on a common probability space $(\mathcal{X},\mathcal{G},\mathbb{Q})$ such that
\begin{enumerate}
\item $(k^{-1}(1-2\bar{M}^{k}_{\lceil tk \rceil})^{-1})_{t\geq0}$  converges almost surely to $(\bar{E}_t^{-1})_{t\geq0}$ in the Skorohod $J_1$ topology as $k\to\infty$.
\item for each $k\in\N$, $(\bar{M}^k_t)_{t\geq0}$ is distributed as $(\tilde{M}_t)_{t\geq0}$
\item $(\bar{E}_t)_{t>0}$ is distributed as $(E_t)_{t>0}$.
\end{enumerate}
Note that item 1 above implies that, for each $\epsilon>0$, $(k(1-2\bar{M}^{k}_{\lceil tk \rceil}))_{t\geq\epsilon}$  converges almost surely to $(\bar{E}_t)_{t\geq\epsilon}$ in the Skorohod $J_1$ topology as $k\to\infty$.

Let $(\bar{b}_i)_{i\in\N}$ be a enumeration of the points of discontinuity of $\bar{E}$ and $\bar{a}_i:=\max\{\bar{b}_j:\bar{b}_j<\bar{b}_i\}$. By the matching of jumps property of the $J_1$ topology, for each $i\in\N$, there exists a sequence $(a^i_k)_{k\in\N}$ such that $a^i_k\to \bar{a}_i$, $k(1-2\bar{M}^{k}_{ ka_k^i })\to\bar{E}_{\bar{a}_i}$ and $k(1-2\bar{M}^{k}_{ ka_k^i- })\to\bar{E}_{\bar{a}_i-}$ as $k\to\infty$. Hence
\be\label{asymptoticbehaviorofpercolationparameter}
\sup_{l\in[k a_k^i, kb_k^i)}\left\vert\bar{M}^{k}_l-\frac{1-k^{-1}\bar{E}_{\bar{a}_i}}{2}\right\vert=o(k^{-1}).
\ee

Let us fix a realization of the processes $(\bar{M}^{k}_t)_{t\geq0}$, $k\in\N$ and $(\bar{E}_t)_{t\geq 0}$. Let $((\bar{I}^i_t)_{t\geq0})_{i\in\N}$ be an independent family of inverse Gaussian subordinators, each one with parameters $\delta=2^{-1/2}$ and $\gamma=\sqrt{2}\bar{E}_{a_i}$. Let $\bar{I}_x:=\sum_{i:\bar{a_i}<x}\bar{I}^i_{\bar{b}_i\wedge x}-\bar{I}^i_{\bar{a}_i}$. Recall that for $p\leq1/2$, $N_p$ denote the size of a percolation tree of parameter $p$. We define
$\bar{V}^{(k)}_x:=\sum_{i=1}^{\lfloor x \rfloor} m^{(k)}_i$, where $(m^{(k)}_i)_{i\in\N}$ is a random variable distributed as $N_{\bar{M}^k_i}$.
Hence, for all $k\in\N$, $(\bar{V}^{(k)}_x)_{x\geq0}$ is distributed as $(V^{\IPC}_x)_{x\geq 0}$ and $(\bar{I}_t)_{t\geq0}$ is distributed as $(I_t)_{t\geq0}$.
Lemma \ref{annealedipc} follows from
\begin{lemma}\label{quenchedipc} We have that, $\mathbb{Q}$-almost surely, 
\[
( k^{-2}\bar{V}^{(k)}_{k x})_{x\geq0}\stackrel{\epsilon\to 0}
    \to (\bar{I}_x)_{x\geq0}
\]
in distribution with the Skorohod $J_1$ topology in $D(\R_+)$.
\end{lemma}
\begin{proof}
We first will prove convergence of marginals and in order to do it we compute Laplace transforms. Let $\delta>0$ be fixed. We can write
\[\E\left[\exp(-\lambda k^{-2}(\bar{V}^{(k)}_{kx}-\bar{V}^{(k)}_{k\delta}))\vert (\bar{M}^k_t)_{t\geq0}\right]
=\prod_{l=\lfloor\delta k\rfloor }^{\lfloor kx \rfloor}\hat{N}_{\bar{M}_l^k}(k^{-2}\lambda)\]
\[=\prod_{\{i:a^k_i\leq x,b^k_i>\delta \}}\prod_{l=\lfloor k(a_i^k\vee \delta)\rfloor}^{\lfloor k (b_i^k\wedge x)\rfloor}\hat{N}_{\bar{M}_l^k}(k^{-2}\lambda).
\]
  By virtue of display (\ref{asymptoticbehaviorofpercolationparameter}), Lemma \ref{laplacetransformofcardinality} and some standard computations we have that
\be\label{convergencetothelaplaceexponentofaninversegaussiansubordinator}
\hat{N}_{\bar{M}_l^k}(k^{-2}\lambda)-1= k^{-1}\left(\bar{E}_{\bar{a}_i}-\sqrt{\bar{E}_{\bar{a}_i}^2+\lambda}\right)+o(k^{-1})
\ee
for all $l\in[k a_k^i, kb_k^i)$ and where the error term is uniform over $l$.
From this it follows that
\[
\lim_{k\to\infty}\prod_{l=\lfloor k(a_i^k\vee \delta)\rfloor}^{\lfloor k (b_i^k\wedge  x)\rfloor}\hat{N}_{\bar{M}_l^k}(k^{-2}\lambda)=\exp\left(((\bar{b}_i\wedge x) -(\bar{a}_i\vee \delta))\left(\bar{E}_{\bar{a}_i}-\sqrt{\bar{E}_{\bar{a}_i}^2+\lambda}\right)\right).
\]
Hence
\[\lim_{k\to\infty}\E[\exp(-\lambda k^{-2}(\bar{V}^{(k)}_{kx}-\bar{V}^{(k)}_{k\delta}))\vert (\bar{M}^k_t)_{t\geq0}]\]
\be\label{marginalsawayfromzero}
=\prod_{\{i:\bar{a}_i\leq x,\bar{b}_i>\delta \}}\exp\left[((\bar{b}_i\wedge x) -(\bar{a}_i\vee \delta))\left(\bar{E}_{\bar{a}_i}-\sqrt{\bar{E}_{\bar{a}_i}^2+\lambda}\right)\right].
\ee
On the other side
\[
\limsup_{\delta\to0}\limsup_{k\to\infty}(1-\E[\exp(-\lambda k^{-2}\bar{V}^{(k)}_{k\delta})\vert (\bar{M}^k_t)_{t\geq0}])
\]
\begin{equation}\label{eq:dispabove3}
\leq\limsup_{\delta\to0}\limsup_{k\to\infty}1-\hat{N}_{\bar{M}^k_{\lfloor\delta k\rfloor}}(k^{-2}\lambda)^{\lfloor \delta k\rfloor}
\end{equation}
because, for any $p_1\geq p_2$ we have that $N_{p_1}$ stochastically dominates $N_{p_2}$ and $\bar{M}^k_t$ is non-decreasing in $t$. Moreover, repeating the computations performed to obtain (\ref{convergencetothelaplaceexponentofaninversegaussiansubordinator}) we get that \eqref{eq:dispabove3} equals
\[
=\limsup_{\delta\to0}1-\exp\left(\delta\left(\bar{E}_{\delta}-\sqrt{\bar{E}_{\delta}^2+\lambda}\right)\right)=0
\]
That plus (\ref{marginalsawayfromzero}) yield that
\[
\lim_{k\to\infty}\E\left[\exp(-\lambda k^{-2}\bar{V}^{(k)}_{kx})\vert (\bar{M}^k_t)_{t\geq0}\right]
\]
\be\label{convergenceofmarginals}
=\prod_{i:\bar{a}_i<x}\exp\left(-(\bar{b}_i\wedge x-\bar{a}_i)(\sqrt{\lambda+\bar{E}_{\bar{a}_i}^2}-\bar{E}_{\bar{a}_i})\right)
\ee
which is the Laplace transform of $\bar{I}_x$. We have proved convergence of marginals.

The convergence of finite-dimensional distributions follows from (\ref{convergenceofmarginals}) and independence. It just remains to show tightness.

 In order to prove tightness we use
\cite[Theorem 15.6]{Billingsley68} which states that the tightness in the $J_1$ topology is implied by
\begin{equation}\label{eq:dispabove4}
\E\left[(k^{-2}\bar{V}^{(k)}_{kx_2}-k^{-2}\bar{V}^{(k)}_{kx})^\beta(k^{-2}\bar{V}^{(k)}_{kx}-k^{-2}\bar{V}^{(k)}_{kx_1})^\beta\middle\vert (\bar{M}^k_t)_{t\geq0}\right]\leq|F(t_2)-F(t_1)|^{2\alpha}
\end{equation}
for $x_1\leq x\leq x_2$ and $k\geq1$ where $\beta\geq0$, $\alpha>1/2$ and $F$ is a nondecreasing, continuous function on $[0,T]$.
By independence, \eqref{eq:dispabove4} is equivalent to
\be\label{tightness}
\E\left[(k^{-2}\bar{V}^{(k)}_{kx_2}-k^{-2}\bar{V}^{(k)}_{kx_1})^\beta\middle\vert (\bar{M}^k_t)_{t\geq0}\right]\leq|F(t_2)-F(t_1)|^{\alpha}.
\ee
But we have that
\[
\E\left[k^{-2}\bar{V}^{(k)}_{kx_2}-k^{-2}\bar{V}^{(k)}_{kx_1}\middle\vert (\bar{M}^k_t)_{t\geq0}\right]\leq(x_2-x_1)\E\left[k^{-1}N_{\bar{M}^k_{kx_2}}\right]
\]
again, because for any $p_1\geq p_2$ we have that $N_{p_1}$ stochastically dominates $N_{p_2}$ and $\bar{M}^k_t$ is non-decreasing in $t$.
But, using display \eqref{convergencetothelaplaceexponentofaninversegaussiansubordinator}, it is easy to see that $\E(k^{-1}N_{\bar{M}^k_{kx_2}})$ converges to $1/2\bar{E}_{x_2}^{-1}$ and so (\ref{tightness}) is satisfied with $\beta=\alpha=1$. Hence we have proved Lemma \ref{quenchedipc}.
\end{proof}

\subsection{Proof of condition $\tilde{L}$}\label{sect_LIPC}

In order to prove Assumption $\tilde{\text{L}}$ we let $\B_p^n$ be a random tree having the law of a percolation cluster on $\mathbb{T}_2^\ast$ of parameter $p$ conditioned on having $n$ vertices. It is not hard to see that the distribution of $\B_p^n$ is uniform over the subtrees of $\mathbb{T}_2^\ast$ having $n$ vertices (that comes from the fact that, for each sub-tree $K$ of $\mathbb{T}_2^\ast$ having $n$ vertices, we have that $\Pb[\B_p^n=K]=p^n(1-p)^{n+1}$). Hence the law of $\B_p^n$ does not depend on $p$. In particular, for any $p\in(0,1)$, the law of $\B_p^n$ equals the law of $\B^n=\B^n_{1/2}$. Then Assumption $\tilde{\text{L}}$ for $X^{\IPC\ast}$ follows from Proposition \ref{assumptionL}.

\subsection{Proof of Theorem~\ref{thm:IPC} and Theorem~\ref{t:descriptionofzipc}}\label{s:proofoftheorem2}
Recall that $(\Lv_x)_{x\in\N}$ denotes the random sequence of branches emerging from the backbone of the IPC. Let $(\Lv_x)_{x\in\Z\setminus\N}$ be a sequence of random trees independent of $(\Lv_x)_{x\in\N}$ and distributed as an i.i.d.~sequence of critical percolation clusters on $\mathbb{T}_2^*$.
Let $X^{\IPC\ast}$ be a randomly trapped random walk with $(\tilde{\nu}[\Lv_x])_{x\in\Z}$ as its random trapping landscape.

Let $\mathbb{F}_1\in M_1(\frak{F}^\ast)$ be as in Proposition \ref{convergenceofWIICast}. Let $(I_x)_{x\geq0}$ be as in Lemma \ref{annealedipc} and $(V_x)_{x\geq0}$ be a $1/2$-stable subordinator independent of $(I_x)_{x\geq0}$
Let \[
I^\ast_x:=\begin{cases}
         I_x &: x\geq0\\
         -V_x &: x< 0.
        \end{cases}
\]
 \begin{proposition}\label{convergenceofWIPCast}
 $(\epsilon X^{\IPC\ast}_{\epsilon^{-3}t})_{t\geq0}$ converges in distribution to  $(B^{\mathbb{F}_1,I^\ast}_t)_{t\geq 0}$ on $(D(\R_+),J_1)$.
\end{proposition}

\begin{proof}
 Assumption $\widetilde{\text{HT}}$ follows from Lemma \ref{annealedipc} and Assumption $\tilde{\text{L}}$ was proved in the previous subsection. This implies the result by Theorem~\ref{RTRWIPC}.
 \end{proof}

It is easy to see that $X^\IPC$ is the restriction of $X^{\IPC\ast}$ to the positive axis. Also $B^+_{\psi^{\IPC}_t}$ is the restriction of $B^{\mathbb{F}_1,I^\ast}_\cdot$ to the positive axis. Hence we can obtain Theorem \ref{thm:IPC} from Proposition \ref{convergenceofWIPCast} in the same way that we obtained Theorem \ref{thm:IIC} from Proposition \ref{convergenceofWIICast}.

\section{Proof of Theorem \ref{prop:alternativeexrepssionforziic} and Theorem \ref{prop:alternativeexrepssionforzipc}}\label{sect_last}

In this section we will prove Theorems \ref{prop:alternativeexrepssionforziic} and \ref{prop:alternativeexrepssionforzipc}. We will start with Theorem \ref{prop:alternativeexrepssionforziic}. Let us make some preliminary definitions to prepare the argument. We recall that $(B_t^{\mathcal{F}})_{t\geq0}$ is the Brownian motion in the Continuum random forest $(\mathcal{F},d,\mu)$ and $(l(x,t))_{x\in{\mathcal{F}},t\geq0}$ is a jointly continuous version of its local time.

Next, we express a reflected Brownian motion in $[0,\infty)$ as a time change of $B^{\mathcal{F}}$.  Let $\lambda$ be the Lebesgue measure on the backbone $[0,\infty)$. Define
\[A_{\text{Bb}}(t):=\int_{\R_+}l(x,t)\lambda(dx)\]
and its right continuous generalized inverse
 \[\tau_{\text{Bb}}(t):=\inf\{s\geq 0:A_{\text{Bb}}(s)>t \}.\]
 It follows from the trace theorem for Dirichlet forms (see Theorem 6.2.1 in \cite{fukushima2010dirichlet} and Lemma 2.4 in \cite{rwrt}) that the time-changed process $B_t^{\text{Bb}}:=B^{\mathcal{F}}_{\tau_{\text{Bb}}(t)}$ is the Brownian motion in the backbone $[0,\infty)$ with respect to the measure $\lambda(dx)$, according to Definition \ref{definitionofbmonadendrite}. Furthermore, it is not hard to see from Definition \ref{definitionofbmonadendrite} that the Brownian motion in the backbone $[0,\infty)$ with respect to the measure $\lambda(dx)$ is simply a standard, reflected Brownian motion. Therefore, $B^{\text{Bb}}$ has the law of a reflected BM.
 Moreover, Lemma 3.4 in \cite{rwrt} states that $l_{\text{Bb}}(x,t):=l(x,\tau(t)), t\geq 0, x\in[0,\infty)$ is a jointly continuous version of the local time of $B_t^{\text{Bb}}$.

It is also possible to construct the BM in each one of branches $\T_i$ of the Continuum Random Forest as a time change of $B^{\mathcal{F}}$. Defining
\[A_i(t):=\int_{\T_i} l(x,t) \bar{y}_i\mu_i(dx) \quad \text{and}\quad \tau_i(t):=\inf\{s\geq0:A_i(s)> t\},\]
again, by the trace theorem for Dirichlet forms we have that $B^{\T_i}_t:=B^{\mathcal{F}}_{\tau_i(t)}$ has the law of the Brownian motion in the branch $(\T_i,\bar{y}_i^{\scriptscriptstyle\frac{1}{2}}d_i,\bar{y}_i\mu_i)$ (according to Definition \ref{definitionofbmonadendrite}) and $l_i(x,t):=l(x,\tau_i(t)),t\geq0, x\in\T_i$ is a jointly continuous version of the local time of $B^{\T_i}$.

Now we state and prove two preliminary lemmas which depend only in the joint continuity of the local times.
 \begin{lemma}
Almost surely, for all $t\geq 0$ and $i\in\N$, there exists a decreasing sequence of times $(s_n)_{n\in\mathbb{N}}$ with 
\begin{equation}\label{eq:dentrodelarama}
\lim_{n\to\infty} s_n=\tau_i(t)\quad\text{ and }\quad B^\mathcal{F}_{s_n}\in\T_i\setminus\{x_0\}.
\end{equation}
Similarly, for all $t\geq 0$, there exists a decreasing sequence of times $(s_n)_{n\in\mathbb{N}}$ with 
\begin{equation}\label{eq:dentrodelacolumna}
\lim_{n\to\infty} s_n=\tau_{\text{Bb}}(t)\quad \text{ and }\quad B^\mathcal{F}_{s_n}\in\R_+\setminus\{x_0\}.
\end{equation}
 \end{lemma}
 \begin{proof}
 We start with the proof of \eqref{eq:dentrodelarama}.
Since $A_i(\cdot)$ is continuous and $\tau_i(t)=\inf\{s>0:A_i(s)>t\}$, there exists a decreasing sequence $(\tilde{s}_n)_{n\in\mathbb{N}}$ with $\lim_{n\to\infty}\tilde{s}_n=\tau_i(t)$ and $A_i(\tilde{s}_n)>A_i(\tilde{s}_{n+1})$ for all $n\in\N$.
Therefore, for all $n\in\mathbb{N}$,  $l_i(\cdot,\tilde{s}_{n+1})>l_i(\cdot,\tilde{s}_n)$ in a set of positive $\mu_i$-measure. Hence, there exists points $x_n\in\T_i\setminus\{x_0\}$ with $l(x_{n},\tilde{s}_{n+1})>l(x_n,\tilde{s}_n)$. This implies that there exists a time $s_n\in(\tilde{s}_n,\tilde{s}_{n+1})$ with $B_{s_n}^\mathcal{F}=x_n$. The sequence $(s_n)_{n\in\N}$ has the desired properties.

The proof of \eqref{eq:dentrodelacolumna} is completely analogous.
 \end{proof}
 
 \begin{lemma}\label{lem:tauiai} Almost surely, for all $s\geq 0$ and $i\in\mathbb{N}$,
\begin{equation}
l(\bar{x}_i,\tau_i(A_i(s)))=l(\bar{x}_i,s).
\end{equation}
\end{lemma}
\begin{proof}
Assume for contradiction that, for some $s>0$,
\[l(\bar{x}_i,\tau_i(A_i(s)))>l(\bar{x}_i,s).\]
Then, given that $l(\bar{x}_i,\cdot)$ is continuous, there exists $r^*\in(s,\tau_i(A_i(s)))$ with
$l(\bar{x}_i,r^\ast)>l(\bar{x}_i,s)$. Therefore, since the local time is continuous in the space variable, it follows that $l(\cdot,r^\ast)>l(\cdot,s)$ at least in a neighborhood of $\bar{x}_i$. Therefore 
\[A_i(r^\ast)=\int_{\T_i} l(x,r^\ast) \bar{y}_i\mu_i(dx)>\int_{\T_i} l(x,s) \bar{y}_i\mu_i(dx)=A_i(s).\]
Hence $r^\ast\geq \inf\{r\geq 0:A_i(r)>A_i(s))\}=\tau_i(A_i(s))$, which is in contradiction with $r^*\in(s,\tau_i(A_i(s)))$. 
\end{proof}
 Having dealt with the preliminary lemmas, we are ready to prove one of the main ingredients for the proofs of Theorems \ref{prop:alternativeexrepssionforziic} and \ref{prop:alternativeexrepssionforzipc}.
\begin{lemma}\label{lem:inverselocaltime}
 Almost surely, for all $t\geq 0$, $i\in\mathbb{N}$,
\[l^{-1}_i(\bar{x}_i,
l_{\text{Bb}}(\bar{x}_i,t))=A_i(\tau_{\text{Bb}}(t))\]
\end{lemma}
\begin{proof}
By definition
\[
l^{-1}_i(\bar{x}_i,l_{\text{Bb}}(\bar{x}_i,t))=\inf\{s\geq 0: l_i(\bar{x}_i,s)>l_{\text{Bb}}(\bar{x}_i,t)\}.
\]
Since $A_i$ is increasing and continuous, we can replace $s=A_i(u)$ in the display above, to obtain
\[
l^{-1}_i(\bar{x}_i,l_{\text{Bb}}(\bar{x}_i,t))=\inf\{A_i(u)\geq 0: l_i(\bar{x}_i,A_i(u))>l_{\text{Bb}}(\bar{x}_i,t)\}.
\]
Again, since $A_i$ is increasing and continuous, we get that
\[
\inf\{A_i(u)\geq 0: l_i(\bar{x}_i,A_i(u))>l_{\text{Bb}}(\bar{x}_i,t)\}=A_i(\inf\{u\geq 0: l_i(\bar{x}_i,A_i(u))>l_{\text{Bb}}(\bar{x}_i,t)\}).
\]
Using the definitions of $l_i$ and $l_{\text{Bb}}$, we get that
\[
A_i(\inf\{u\geq 0: l_i(\bar{x}_i,u)>l_{\text{Bb}}(\bar{x}_i,t)\})=A_i(\inf\{u\geq 0: l(\bar{x}_i,\tau_i(A_i(u)))>l(\bar{x}_i,\tau_{\text{Bb}}(t))\}).
\]
Using Lemma \ref{lem:tauiai}, we get that
\[
A_i(\inf\{u\geq 0: l(\bar{x}_i,\tau_i(A_i(u)))>l(\bar{x}_i,\tau_{\text{Bb}}(t))\}=A_i(\inf\{u\geq 0: l(\bar{x}_i,u)>l(\bar{x}_i,\tau_{\text{Bb}}(t))\}.
\]
Hence
\begin{equation}\label{eq:identitytorefer}
l^{-1}_i(\bar{x}_i,l_{\text{Bb}}(\bar{x}_i,t))=A_i(\inf\{u\geq 0: l(\bar{x}_i,u)>l(\bar{x}_i,\tau_{\text{Bb}}(t))\}.
\end{equation}
Therefore, it suffices to show that, almost surely
\begin{equation}\label{eq:2ineq}
\begin{aligned}
A_i(\inf\{u\geq 0: l(\bar{x}_i,u)>l(\bar{x}_i,\tau_{\text{Bb}}(t))\})=A_i(\tau_{\text{Bb}}(t))\end{aligned}
\end{equation}
for all $t\geq0$.

Let $t>0$ fixed and $\theta_{\bar{x}_i}(\tau_{\text{Bb}}(t)):=\inf\{s>\tau_{\text{Bb}}(t):B^{\mathcal{F}}_s=\bar{x}_i\}$.  First, $A_i$ increases only when $B^{\mathcal{F}}$ is at $\T_i$. Also, at time $\tau_{\text{Bb}}(t)$, $B^{\mathcal{F}}$ is at the backbone. Therefore, if $s\in[\tau_{\text{Bb}}(t),\theta_{\bar{x}_i}(\tau_{\text{Bb}}(t))]$ then $B^\mathcal{F}_s\not\in \mathcal{T}_i$. Therefore, $A_i$ cannot increase between $\tau_{\text{Bb}}(t)$ and $\theta_{\bar{x}_i}(\tau_{\text{Bb}}(t))$.
Hence
\[
A_i(\tau_{\text{Bb}}(t))=A_i(\theta_{\bar{x}_i}(\tau_{\text{Bb}}(t)))
\]
Therefore, to finish the proof of the lemma, it is enough to show that, almost surely
\begin{equation}\label{eq:newproof}
\theta_{\bar{x}_i}(\tau_{\text{Bb}}(t))=\inf\{u\geq 0: l(\bar{x}_i,u)>l(\bar{x}_i,\tau_{\text{Bb}}(t))\}
\end{equation}
for all $t\geq0$.
Moreover, since, for each fixed $t$, $\tau_{\text{Bb}}(t)$ is a stopping time (for $B^{\mathcal{F}}$), it follows that $\theta_{\bar{x}_i}(\tau_{\text{Bb}}(t))$ is also a stopping time. 
Therefore, by the strong Markov property of the Brownian motion in the continuum random forest, we have that the local time at $\bar{x}_i$ increases immediately after $\theta_{\bar{x}_i}(\tau_{\text{Bb}}(t))$. That is \begin{equation}
l(\bar{x}_i,\theta_{\bar{x}_i}(\tau_{\text{Bb}}(t)))<l(\bar{x}_i,s^*),\quad \forall s^*>\theta_{\bar{x}_i}(\tau_{\text{Bb}}(t)).
\end{equation}
This shows \eqref{eq:newproof} for a fixed $t$. To have the display for all $t\geq 0$, it suffices to note that both, the left hand side and the right hand side are right-continuous functions of $t$.
\end{proof}
We also need to guarantee independence of the processes involved.
\begin{lemma}\label{lem:independence}
The family of processes $B^{\text{Bb}},B^{\T_i},i\in\mathbb{N}$ is jointly independent.
\end{lemma}
\begin{proof}
	We will start showing that $B^{\T_i}$ is independent of $B^{\T_j}$ for $i\neq j$.
	Let $a,b$ points in the backbone $\bar{x}_i<a<b<\bar{x}_j$ (where we are assuming without loss of generality, that $\bar{x}_i<\bar{x}_j$).
	Let $\theta_0=0$,
	\[\theta_1:= \inf\{s\geq 0: B^{\mathcal{F}}_s=b\}\]
	and
	\[\theta_{2k}:=\inf\{s>\theta_{2k-1}:B^{\mathcal{F}}_s=a\},\]
	\[
	\theta_{2k+1}:=\inf\{s>\theta_{2k}:B^{\mathcal{F}}_s=b\} 
	\]
	for $k\in\N$.
	
Then, one has that $B^{\mathcal{T}_i}$
and $B^{\mathcal{T}_j}$ depend upon disjoint intervals of time. More precisely, $B^{\mathcal{T}_i}$ depends on $\cup_{k\geq0} [\theta_{2k},\theta_{2k+1})$ and $B^{\mathcal{T}_i}$ depends on $\cup_{k\geq 0} [\theta_{2k+1},\theta_{2k+2})$. Therefore, the strong Markov property of $B^{\mathcal{F}}$ at the stopping times $\theta_k,k\in\mathbb{N}$ gives the desired independence.
 
 Now, we will show that, for all $i\in\mathbb{N}$, $B^{\text{Bb}}$ and $B^{\mathcal{T}_i}$ are independent.
 We will use the following property of the CRT: For each $n\in\N$ (sufficiently large), there exists a unique point $a_n$ at distance $1/n$ from the root such that the volume of the descendants of $a_n$ is larger that $1/2$ (any constant value would work). This can be shown from the excursion representation of the CRT, we omit the proof. 
 Moreover, it follows from said representation that
 \begin{equation}\label{eq:vanishingmeasure}
\mu_i(\{x\in\T_i: x \nsucceq a_n\})\stackrel{n\to\infty}{\to} 0
 \end{equation} 
 almost surely, where we recall that $\prec$ denotes genealogical order.
  
Let $\theta^{\bar{x}_i,\text{in}}_0=0$ and $\theta^{\bar{x}_i,\text{out}}_0=\inf\{s\geq 0: B^\mathcal{F}_s= a_n\}$.
For $k\in\N$,  
\[\theta^{\bar{x}_i,\text{in}}_k=\inf\{s\geq \theta^{\bar{x}_i,\text{out}}_{k-1}: B^\mathcal{F}_s=\bar{x}_i\},\]
\[
\theta^{\bar{x}_i,\text{out}}_k=\inf\{s\geq \theta^{\bar{x}_i,\text{in}}_k: B^\mathcal{F}_s=a_n\}
\]
Let also $\theta^{a_n,\text{in}}_0=\inf\{s\geq 0: B^\mathcal{F}_s=a_n\}$.
For $k\in\N$,  
\[\theta^{a_n,\text{in}}_k=\inf\{s\geq \theta^{a_n,\text{out}}_{k-1}: B^\mathcal{F}_s=a_n\},\]
\[
\theta^{a_n,\text{out}}_k=\inf\{s\geq \theta^{a_n,\text{in}}_k: B^\mathcal{F}_s=\bar{x}_i\}.
\]

It follows that the intervals $[\theta^{\bar{x}_k,\text{in}}_k,\theta^{\bar{x}_k,\text{out}}_k),[\theta^{a_n,\text{in}}_k,\theta^{a_n,\text{out}}_k)$, $k\in\mathbb{N}$ are pairwise disjoint.
It is not hard to show that $B^{+}$ depends only on the time intervals $[\theta^{\bar{x}_i,\text{in}}_k,\theta^{\bar{x}_i,\text{out}}_k), k\in\N$.

Let us define $A^n(t):=\int_{\{x:x\succeq a_n\}}l(x,t)\lambda(dx)$ and
$\tau^n(t):=\inf\{s\geq 0:A^n(s)>t \}$.
Let also
$B^{\T_i,n}_t:=B^{\mathcal{F}}_{\tau^n(t)}$.
It is not hard to show that $B^{\T_i,n}_t$ depends only on the time intervals $[\theta^{a_n,\text{in}}_i,\theta^{a_n,\text{out}}_i), i\in\N$.
Therefore, using the Strong Markov property of $B^{\mathcal{F}}$ at the stopping times $\theta^{a_n,\text{in}}_k,\theta^{a_n,\text{out}}_k,\theta^{\bar{x}_i,\text{in}}_k,\theta^{\bar{x}_i,\text{out}}_k, k\in\N$, it is possible to show that $B^{\T_i,n}$ and $B^{\text{Bb}}$ are pairwise independent.

Moreover, using \eqref{eq:vanishingmeasure}, it can be shown that $A^n_i$ converges almost surely (and uniformly over compact intervals of time) to $A_i$. It follows that $\tau_i^n$ converges to $\tau_i$. This, together with the uniform continuity of $B^{\mathcal{F}}$, implies that $B^{\T_i,n}$ converges to $B^{\T_i}$. Hence, $B^{\T_i,n}$ and $B^{\text{Bb}}$ are independent, it follows that $B^{\T_i}$ and $B^{\text{Bb}}$ are independent. 

The same reasoning can be generalized to show the joint independence for any finite number of processes in the family $B^{\text{Bb}},B^{\T_i},i\in\mathbb{N}$. This finishes the proof.
\end{proof}

 Recall the definition of $\phi^{\IIC}$ from \eqref{eq:defphiiic}. Putting together the last two lemmas, we get
\begin{lemma}\label{lem:quid}
	 The process $(B^{\text{Bb}}_{A_{\text{Bb}}(t)})_{t\geq0}$ has the same distribution as $(Z^{\IIC}_t)_{t\geq 0}$.
\end{lemma}
\begin{proof}
 Let us express the time change $A_{\text{Bb}}$ as the \emph{inverse of its inverse}, that is $A_{\text{Bb}}(t)=\inf\{s\geq 0:\tau_{\text{Bb}}(s)>t\}$, 
which follows since $A_{\text{Bb}}$ is non-decreasing and $A_{\text{Bb}}(0)=0$ (see \cite[Corollary 13.6.1]{whi02}).
 Hence, recalling the definition $Z^\IIC_t:=B^+_{\psi^{\IIC}_t}$ (with $\psi^{\IIC}:=(\phi^\IIC)^{-1}$), we see that it suffices to show that the pair $(B^{\text{Bb}}_t,\tau_\text{Bb}(t))$ has the same distribution as $(B^+_t,\phi^\IIC_t)$. We already know that $B^{\text{Bb}}$ and $B^{+}$ have the same distribution, namely, that of a reflected Brownian motion. It remains to show that $\tau_\text{Bb}(t)$ can be constructed from $B^{\text{Bb}}$ in the same fashion that $\phi^\IIC_t$ is constructed from $B^+$, this is, by summing independent processes (having the annealed law of the inverse local time at the root of the BM on the CRT) evaluated at the local time of $B^+$ in the points $\bar{x}_i$.
 We can write
\begin{equation}\label{eq:tracethm}
\begin{aligned}
\tau_{\text{Bb}}(t)=&\int_{\mathcal{F}}l(x,\tau_{\text{Bb}}(t))\mu(dx)\\
=&\sum_{i\in\N}\int_{\T_i}l(x,\tau_{\text{Bb}}(t))\bar{y}_i\mu_i(dx).
\end{aligned}
\end{equation}
where we recall that $(\T_i)_{i\in\N}$ denotes the collection of branches of the CRF and we are using the fact that the backbone has $\mu$-measure $0$.
For each $i$, the $i$-th summand in the display above is equal to $A_i(\tau_{\text{Bb}}(t))$. Therefore, by Lemma \ref{lem:inverselocaltime} and  \eqref{eq:tracethm} 
\be\label{eq:alternativedescription}
\tau_{\text{Bb}}(t)=\sum_{i\in\N} l^{-1}_i(\bar{x}_i,l_{\text{Bb}}(\bar{x}_i,t)).
\ee
 By simple scaling properties of the CRT, the law of $(\bar{y}_i^{-2/3}l^{-1}_i(\bar{x}_i,\bar{y}_i^{2}s))_{s\geq0}$ (when regarding $\T_i$ as a random object) has the annealed law of the inverse local time at the root of the BM on the CRT. Hence, 
recalling that the processes $S^i$ appearing in the definition of $\phi^\IIC$ have the annealed law of the inverse local time at the root of the BM on the CRT, we get that
 \begin{equation}\label{eq:lastformula}
 (l^{-1}_i(\bar{x}_i,s))_{s\geq0}\stackrel{d}{=}(\bar{y}_i^{-3/2}S^i(\bar{y}_i^{\scriptscriptstyle\frac{1}{2}}s))_{s\geq0},
 \end{equation}
 where $\stackrel{d}{=}$ denotes equality in distribution. 
 Finally, since for each $i\in\N$ the inverse local time process $(l^{-1}_i(\bar{x}_i,t))_{t\geq0}$ can be constructed $B^{\T_i}$ we can apply Lemma \ref{lem:independence} to get the independence between the processes $l^{-1}_i(\bar{x}_i,\cdot),i\in\N$ and $l_{\text{Bb}}(\cdot,\cdot)$. Therefore, displays \eqref{eq:alternativedescription} and \eqref{eq:lastformula} together with the fact that $(l_{\text{Bb}}(x,t))_{ x\in\R_+,t\geq0}$ is the local time of a reflected Brownian motion in $\R_+$ (and the independence guaranteed by Lemma \ref{lem:independence}) gives that
\[
(\phi^{\IIC}_t)_{t\geq0}\stackrel{d}{=}(\tau_{\text{Bb}}(t))_{t\geq0}.
\] 
  Therefore, in \eqref{eq:alternativedescription} the process $\tau_{\text{Bb}}$ is constructed from the local time of $B^{\text{Bb}}$ in the same way that $\phi^\IIC_t$ is constructed from the local time of $B^+$. This, together with the fact that the reflected BM, $B^{\text{Bb}}$ is independent of the processes $l_i^{-1}(\bar{x}_i,\cdot),i\in\N$ (guaranteed by Lemma \ref{lem:independence}), establishes that $(B^{\text{Bb}}_t,\tau_\text{Bb}(t))$ has the same distribution as $(B^+_t,\phi^\IIC_t)$.


\end{proof}

We are ready to prove Theorem 
\ref{prop:alternativeexrepssionforziic}.
\begin{proof}[Proof of Theorem \ref{prop:alternativeexrepssionforziic}]
Let $\theta_{\text{Bb}}(t):=\inf\{s> t: B^{\mathcal{F}}_s\text{ is in the backbone}\}$. 
By Lemma \ref{lem:quid}, it suffices to establish the following chain of identities:
\[B^{\text{Bb}}_{A_{\text{Bb}}(t)}=B^\mathcal{F}_{\tau_{\text{Bb}}(A_{\text{Bb}}(t))}=B^{\mathcal F}_{\theta_{\text{Bb}}(t)}=\pi(B^\mathcal{F}_t),\]
where $\stackrel{d}{=}$ denotes equality in law.
The first equality follows from the definition of $B^{\text{Bb}}$. The second equality
 will follow after we have showed that, almost surely,
\begin{equation}\label{eq:third}
\tau_{\text{Bb}}(A_{\text{Bb}}(t))=\theta_{\text{Bb}}(t)\quad \forall t\geq 0.
\end{equation}
The last equality is a consequence of the trivial fact that, for a continuous trajectory on a tree, the next hitting point of a subtree coincides with the current projection over that subtree.

Now we deal with the proof of \eqref{eq:third}.
It suffices to show that,
\begin{equation}\label{eq:third1}
\forall s >\theta_{\text{Bb}}(t), \quad A_{\text{Bb}}(s)>A_{\text{Bb}}(t)
\end{equation} 
and 
\begin{equation}\label{eq:third2}
\forall s <\theta_{\text{Bb}}(t), \quad A_{\text{Bb}}(s)\leq A_{\text{Bb}}(t).
\end{equation}
For the proof of \eqref{eq:third1}, we notice that, for all $t$ fixed, $\theta_{\text{Bb}}(t)$ is a stopping time. Therefore, almost surely,
\begin{equation}
l(B^{\mathcal{F}}_{\theta_{\text{Bb}}(t)},\theta_{\text{Bb}}(t))<l(B^{\mathcal{F}}_{\theta_{\text{Bb}}(t)},s^*),\quad \forall s^*>\theta_{\text{Bb}}(t).
\end{equation}
Moreover, since  the local time is continuous in the space variable, it follows from the display above that $l(\cdot,\theta_{\text{Bb}}(t))<l(\cdot,s^*)$ in an open set of the backbone, for all $s^*>\theta_{\text{Bb}}(t)$.
Hence $A_{\text{Bb}}(s^
*)>A_{\text{Bb}}(\theta_{\text{Bb}}(t))\geq A_{\text{Bb}}(t)$.
Therefore, for all $t\geq 0$
\begin{equation}
\mathbb{P}[\tau(A_{\text{Bb}}(t))=\theta_{\text{Bb}}(t)]=1.
\end{equation}
and, since both $\theta_{\text{Bb}}(\cdot)$ and $\tau_{\text{Bb}}(A_{\text{Bb}}(\cdot))$ are right continuous, display \eqref{eq:third} follows. This finishes the proof.
\end{proof}

The proof follows the same argument as that of Theorem \ref{prop:alternativeexrepssionforziic}.
Recall from ... that $B^{\tilde{\mathcal{F}}}$ denotes the Brownian motion in the modified forest $(\tilde{\mathcal{F}},\tilde{d},\tilde{\mu})$. 
Let $\tilde{l}(t,x)$ be a jointly continuous version of the local time of $B^{\tilde{\mathcal{F}}}$ and $\lambda$ denote the Lebesgue measure on $[0,\infty)$. 
Define
\[\tilde{A}_{\text{Bb}}(t):=\int_{\R_+}\tilde{l}(x,t)\lambda(dx) \quad \text{and}\quad \tilde{\tau}_{\text{Bb}}(t):=\inf\{s\geq 0:\tilde{A}_{\text{Bb}}(s)>t \}.\]
As in the case of $B^{\text{Bb}}$, we have that $\tilde{B}^{\text{Bb}}_t:=B^{\tilde{\mathcal{F}}}_{\tilde{\tau}_{\text{Bb}}(t)}$ is a reflected Brownian motion.
Let
\[\tilde{A}_i(t):=\int_{\T_i} l(x,t) \tilde{y_i}\mu_i(dx) \quad \text{and}\quad \tilde{\tau}_i(t):=\inf\{s\geq0:\tilde{A}_i(s)> t\},\]
again, by the trace theorem for Dirichlet forms we have that $\tilde{B}^{\T_i}_t:=B^{\tilde{\mathcal{F}}}_{\tilde{\tau}_i(t)}$ has the law of the Brownian motion in the branch $(\T_i,\tilde{y}_i^{\scriptscriptstyle\frac{1}{2}}d_i,\tilde{y}_i\mu_i)$ 
and $\tilde{l}_i(x,t):=\tilde{l}(x,\tilde{\tau_i}(t)),t\geq0, x\in\T_i$ is a jointly continuous version of the local time of $\tilde{B}^{\T_i}$.
\begin{lemma}\label{lem:shame}
The family of processes $\tilde{B}^{\text{Bb}},\tilde{B}^{\T_i},i\in\mathbb{N}$ is jointly independent.
\end{lemma}
\begin{proof}
The proof is completely analogous to that of Lemma \ref{lem:independence}.
\end{proof}

\begin{lemma}\label{eq:analogousidentity}
 Almost surely, for all $t\geq 0$, $i\in\mathbb{N}$,
\[
\tilde{l}^{-1}_i(\tilde{x}_i,
\tilde{l}_{\text{Bb}}(\tilde{x}_i,t))=\tilde{A}_i(\tilde{\tau}_{\text{Bb}}(t)).
\]
\end{lemma}
\begin{proof}
The proof is completely analogous to that of Lemma \ref{lem:inverselocaltime}.
First we get that, 
almost surely, for all $s\geq 0$ and $i\in\mathbb{N}$,
\begin{equation}\label{eq:tildetauiai}
\tilde{l}(\tilde{x}_i,\tilde{\tau}_i(\tilde{A}_i(s)))=\tilde{l}(\tilde{x}_i,s).
\end{equation}
in the same way we have obtained Lemma \ref{lem:tauiai}. Indeed, the proof of Lemma \ref{lem:tauiai} uses only the joint continuity of the local time, which also holds for $\tilde{l}(\cdot,\cdot)$.
Furthermore, we can repeat verbatim the computations leading to \eqref{eq:identitytorefer} (replacing every instance of $l_i(\cdot,\cdot), \bar{x}_i, l_{\text{Bb}}(\cdot,\cdot),A_i$ by $\tilde{l}_i(\cdot,\cdot), \tilde{x}_i, \tilde{l}_{\text{Bb}}(\cdot,\cdot),\tilde{A}_i$ respectively and using \eqref{eq:tildetauiai} instead of Lemma \ref{lem:tauiai}) to obtain that, almost surely
\begin{equation}
\begin{aligned}
\tilde{l}^{-1}_i(\tilde{x}_i,\tilde{l}_{\text{Bb}}(\tilde{x}_i,t))=\tilde{A}_i(\inf\{u\geq 0: \tilde{l}(\tilde{x}_i,u)>\tilde{l}(\tilde{x}_i,\tilde{\tau}_{\text{Bb}}(t))\})
\end{aligned}
\end{equation}
for all $t\geq0$. Therefore, it is enough to show that, almost surely
\begin{equation}\label{eq:pleasestop}
\begin{aligned}
\tilde{A}_i(\inf\{u\geq 0: \tilde{l}(\tilde{x}_i,u)>\tilde{l}(\tilde{x}_i,\tilde{\tau}_{\text{Bb}}(t))\})=\tilde{A}_i(\tilde{\tau}_{\text{Bb}}(t))\end{aligned}
\end{equation}
for all $t\geq0$.
Let
$\tilde{\theta}_{\tilde{x}_i}(\tilde{\tau}_{\text{Bb}}(t)):=\inf\{s>\tilde{\tau}_{\text{Bb}}(t):B^{\tilde{\mathcal{F}}}_s=\tilde{x}_i\}$.
As in the proof Lemma \ref{lem:inverselocaltime}, we have that $\tilde{A}_i$ cannot increase between $\tilde{\tau}_{\text{Bb}}(t)$ and $\tilde{\theta}_{\tilde{x}_i}(\tilde{\tau}_{\text{Bb}}(t))$. Therefore, we get that
\begin{equation}\label{eq:pleasestop2}
\tilde{A}_i(\tilde{\tau}_{\text{Bb}}(t))=\tilde{A}_i(\tilde{\theta}_{\tilde{x}_i}(\tilde{\tau}_{\text{Bb}}(t)))
\end{equation}
In the same way we obtained \eqref{eq:newproof}, that is, using the strong Markov property of $B^{\tilde{\mathcal{F}}}$ to deduce that its local time $\tilde{l}(\tilde{x}_i,\cdot)$ increases immediately after the stopping time $\tilde{\theta}_{\tilde{x}_i}(\tilde{\tau}_{\text{Bb}}(t))$, we get that almost surely
\begin{equation}\label{eq:pleasestop3}
\tilde{\theta}_{\tilde{x}_i}(\tilde{\tau}_{\text{Bb}}(t))=\inf\{u\geq 0: \tilde{l}(\tilde{x}_i,u)>\tilde{l}(\tilde{x}_i,\tilde{\tau}_{\text{Bb}}(t))\}
\end{equation}
for all $t\geq0$. Displays \eqref{eq:pleasestop2} and \eqref{eq:pleasestop3} yield \eqref{eq:pleasestop}. This finishes the proof.

\end{proof}

\begin{lemma}\[
(\tilde{B}^{\text{Bb}}_{\tilde{A}_{\text{Bb}}(t)})_{t\geq0}\stackrel{d}{=}(Z^{\IPC}_t)_{t\geq0}.
\]
\end{lemma}
\begin{proof}
The proof is completely analogous to that of Lemma \ref{lem:quid}, with the only difference being that we replace the point process $(\bar{x}_i,\bar{y}_i)$ used to choose the locations and sizes of the branches of the Continuum Random Forest should be replaced by $(\tilde{x}_i,\tilde{y}_i)$. We start by writing
\begin{equation}
\begin{aligned}
\tilde{\tau}_{\text{Bb}}(t)=&\int_{\tilde{\mathcal{F}}}l(x,\tilde{\tau}_{\text{Bb}}(t))\tilde{\mu}(dx)\\
=&\sum_{i\in\N}\int_{\T_i}l(x,\tilde{\tau}_{\text{Bb}}(t))\tilde{y}_i\mu_i(dx)\\
=&\sum_{i\in\N} \tilde{A}_i(\tilde{\tau}_{\text{Bb}}(t)).
\end{aligned}
\end{equation}
Therefore, using Lemma \ref{eq:analogousidentity} we get that, almost surely,
\[\tilde{\tau}_{\text{Bb}}(t)=\sum_{i\in\N}\tilde{l}^{-1}_i(\tilde{x}_i,
\tilde{l}_{\text{Bb}}(\tilde{x}_i,t))\]
for all $t\geq0$.
Recalling that the $S^i$ are independent and distributed according the annealed law of the inverse local time of the Brownian motion on the CRT, in the same way we obtained \eqref{eq:lastformula}, we get that 
\begin{equation}\label{eq:lastformula}
 (\tilde{l}^{-1}_i(\tilde{x}_i,s))_{s\geq0}\stackrel{d}{=}(\tilde{y}_i^{-3/2}S^i(\tilde{y}_i^{\scriptscriptstyle\frac{1}{2}}s))_{s\geq0},
 \end{equation}
 where we regard the inverse local time averaged with respect to the randomness of the branch $\T_i$. Moreover, by Lemma \ref{lem:shame}, we have that the processes $(\tilde{l}^{-1}_i(\tilde{x}_i,s))_{s\geq0},i\in\N$ are independent between them and also independent from $\tilde{l}_{\text{Bb}}(\cdot,\cdot)$, we get that 
\[
(\phi^{\IPC}_t)_{t\geq0}\stackrel{d}{=}(\tilde{\tau}_{\text{Bb}}(t))_{t\geq0}.
\]
 Finally, using again the independence between the family of processes $(\tilde{l}^{-1}_i(\tilde{x}_i,s))_{s\geq0},i\in\N$ and $\tilde{B}^{\text{Bb}}$, guaranteed by Lemma \ref{lem:shame}, we get that $(\tilde{B}^{\text{Bb}},\tilde{\tau}_{\text{Bb}})$ have the same distribution as $(B^+,\phi^{\IPC})$, and, as in the proof of Theorem 
\ref{prop:alternativeexrepssionforziic}, we get that $(\tilde{B}^{\text{Bb}},\tilde{A}_{\text{Bb}})$ have the same distribution as $(B^+,\psi^{\IPC})$. Recalling that $Z^\IPC_t:=B^+_{\psi^{\IPC}_t}$, we have that $Z^{\IPC}_t$ has the same distribution as $\tilde{B}^{\text{Bb}}_{\tilde{A}_{\text{Bb}}(t)}$.
\end{proof}
\begin{proof}[Proof of Theorem \ref{prop:alternativeexrepssionforzipc}]
We can proceed in the same way as in Theorem \ref{prop:alternativeexrepssionforziic}, the only difference is in the construction of the point process $(\tilde{x}_i,\tilde{y}_i)_{i\in\N}$, but this does not affect the proof.
Let $\tilde{\theta}_{\text{Bb}}(t):=\inf\{s> t: B^{\tilde{\mathcal{F}}}_s\text{ is in the backbone}\}$. 
It is enough to show that
\[\tilde{B}^{\text{Bb}}_{\tilde{A}_{\text{Bb}}(t)}=B^{\tilde{\mathcal{F}}}_{\tilde{\tau}_{\text{Bb}}(\tilde{A}_{\text{Bb}}(t))}=B^{\tilde{\mathcal{F}}}_{\tilde{\theta}_{\text{Bb}}(t)}=\pi(B^{\tilde{\mathcal{F}}}_t),\]
where $\stackrel{d}{=}$ denotes equality in law.
The first equality follows from the definition of $\tilde{B}^{\text{Bb}}$. The last equality is a consequence of the trivial fact that, for a continuous trajectory on a tree, the next hitting point of a subtree coincides with the current projection over that subtree.
 As in the proof of Theorem \ref{prop:alternativeexrepssionforziic}, the second equality
will follow after we have showed that, almost surely,
\begin{equation}
\tilde{\tau}_{\text{Bb}}(\tilde{A}_{\text{Bb}}(t))=\tilde{\theta}_{\text{Bb}}(t)\quad \forall t\geq 0,
\end{equation}
but the proof of this identity is completely analogous to that of \eqref{eq:third}.
\end{proof}

  We would like to finish this section mentioning some results that are related to the convergence of the IIC to $\mathcal{F}$ and the IPC to $\tilde{\mathcal{F}}$. In \cite{AngelGoodmanMerle2013} the scaling limit of the search-depth processes of IIC and the IPC are identified. The trees $\mathcal{F}$ and $\tilde{\mathcal{F}}$ should be obtained from the limiting search depth processes in the same way that the CRT is obtained from the normalized Brownian excursion.

\section{Finite versions of the SSBM and their link to the Brownian motion on the CRT}\label{sect_finite_SSBM}

In this section, we will define the Brownian motion on the $K$-reduced tree of a CRT and then propose an alternative construction of this process as an SSBM.

\subsection{The Brownian motion projected onto the $K$-reduced tree of a CRT}
\label{sect_BCRT}

Proving convergence towards the Brownian motion (of some relevant discrete model) on the CRT is an important problem with applications in the study of the simple random walk on critical  trees and critical graphs in $\Z^d$ in high dimensions such as critical percolation, lattice trees, critical branching random walks among other models.

A natural approach for proving this convergence is to show a finite dimensional version of it along with some tightness. More precisely, one defines a reduced tree on the discrete model and study the scaling properties of the random walk projected onto the reduced sub-tree. 
Here is where extensions of the SSBM on finite trees are useful because, as we will see, they can represent the continuous analogous of the process above, i.e., the BM on the CRT projected to the backbone.

Consider the CRT $\mathfrak{T}$, which comes with a uniform measure. We can chose $K$ random uniform points and build from those points the $K$-reduced tree $\mathfrak{T}^{(K)}$ as in \eqref{eq:defofreduced}. Let us then define $\pi_{\mathfrak{T}^{(K)}}:\mathfrak{T}\to\mathfrak{T}^{(K)}$ the projection onto $\mathfrak{T}^{(K)}$. That is, for any $x\in\mathfrak{T}$, $\pi_{\mathfrak{T}^{(K)}}(x)$ is the point in $\mathfrak{T}^{(K)}$ which is the closest to $x$ according to the natural distance on $\mathfrak{T}$. 

Using the notation $B^{\mathfrak{T}}$ for the Brownian motion on the CRT, we denote 
\[
Z^{K\text{-crt}}_t:=\pi_{\mathfrak{T}^{(K)}}(B^{\mathfrak{T}}_t).
\]

This yields a stochastic process on the  finite tree $\mathfrak{T}^{(K)}$. One of the central ideas in \cite{rwrt} to prove convergence to the BM on the CRT was to approximate $B^{\text{CRT}}$ by $B^{K\text{-crt}}$ ($K$ large).

 In \S \ref{eq:aftermuerto}, we will show how to express  $B^{K\text{-crt}}$ as an SSBM on $\mathfrak{T}^{(K)}$ (More precisely, as an SSBM where the Poisson point process $(x_i,y_i)_{i\in\N}$ is conditioned on $\sum_{i\in\N} y_i=1$).
 
\subsection{The Brownian motion projected onto the $K$-reduced tree of the CRT as a finite SSBM}

One of the points that we would like to stress with this construction is that the relation between SSBMs and the BM on the CRT is twofold. One one hand, as we have anticipated, the projection of the BM on the reduced sub-trees is an SSMB.  On the other hand, the BM on the CRT can be seen as a limit of SSBMs on reduced sub-trees, which can be built independently from the CRT through the so called \emph{line breaking construction}.
    
\subsubsection{The line-breaking construction}

Next, we recall an alternative construction of $\mathfrak{T}^{(K)}$ introduced by Aldous in \cite{AldousCRT1}. This construction can be relevant in practice because it shows that the $K$-CRT (and hence the Brownian motion on the CRT) can be constructed in a relatively elementary manner that does not require a full description of the CRT itself.

Let $(C_1,C_2,\ldots)$ be the times of and inhomogeneous Poisson process on $(0,\infty)$ with rate $r(t)=t$. Let $\mathcal{R}(1)$ consist of an edge of length $C_1$ from a root to the leaf $1$. Then, inductively we can obtain $\mathcal{R}(k+1)$ from $\mathcal{R}(k)$ by attaching an edge of length $C_{k+1}-C_k$ to a uniform random point of $\mathcal{R}(k)$.

It is known (see the proof of Lemma 21, and the paragraph following Corollary 22 in \cite{AldousCRT3}) that this construction yields a tree that has the same law as $\mathfrak{T}^{(K)}$.

\subsubsection{Construction of $\mathfrak{T}$ from $\mathfrak{T}^{(K)}$}\label{sect_this_sucks}

In this section, we show how to build the CRT from $\mathfrak{T}^{(K)}$ and a Brownian bridge conditioned on local time.

Let us explain how to attach branches to $\mathfrak{T}^{(K)}$ to get $\mathfrak{T}$. The branches that hang off $\mathfrak{T}^{(K)}$ (i.e., the connected components of $\mathfrak{T}\setminus\mathfrak{T}^{(K)}$) are a countable collection of (scaled) CRTs which are independent except for the fact that their total volume is conditioned to be $1$ (because $\mathfrak{T}$ has total volume $1$ and $\mathfrak{T}^{(K)}$ has zero volume).
To construct such sequence of branches, we will use $(\overline{B}^{C_K}_t)_{t\in [0,1]}$ a reflected Brownian bridge reaching $0$ at time $1$ conditioned on having total local time at $0$ equal to $C_K$ (where we recall that $C_K$ is the total length of $\mathfrak{T}^{(K)}$). This stochastic process is chosen independently of the random variables of the previous section. For a rigorous definition of the reflected Brownian bridge conditioned on local time  we refer to \cite{pitman1999sde} and \cite{chassaing2001vervaat}. We denote by $(\overline{L}^{C_K}_t)_{t\in[0,1]}$ the local time at the origin of $\overline{B}^{C_K}$.

Next, we will decompose the Brownian bridge through excursions.
Let $(d_i)_{i\in\N}$ be an enumeration of the discontinuities of the inverse local time $(\overline{L}^{C_K})^{-1}$ (which will range from $0$ to $C_K$) and $\overline{e}_i$ the corresponding excursions, i.e., the function defined for  $t\in[0,(\overline{L}^{C_K})^{-1}_{d_i}-(\overline{L}^{C_K})^{-1}_{d_i^-}]$ which is equal to $ \overline{B}^{C_K}_{(\overline{L}^{C_K})^{-1}_{d_i^-}+t}$.

Using the procedure to construct trees from excursions (see Section \ref{s:crt}) we can use the $\overline{e}_i$ to construct scaled CRTs: The construction in \S \ref{s:crt} can be generalized for excursions whose durations are different from $1$. This yields trees with volume different from $1$. Therefore, the real trees $(\mathfrak{T}_{\overline{e}_i})_{i\in\N}$ constructed from $\overline{e}_i$ are scaled CRTs instead of the usual, normalized CRTs.  
The $(\mathfrak{T}_{\overline{e}_i})_{i\in\N}$ is the desired sequence of (scaled) CRTs conditioned in that their volumes sum up to $1$.                          
 
 The next step is to identify where in $\mathfrak{T}^{(K)}$ are the branches $\mathfrak{T}_{\overline{e}_i}$.
 It is clear that we can parametrize the tree $\mathfrak{T}^{(K)}$ (using its length) by a function $\rm{F}_{\mathfrak{T}^{(K)}}: [0,C_K] \to \mathfrak{T}^{(K)}$. Recall that $(d_i)_{i\in\mathbb{N}}$ are the points of discontinuity of the inverse local time of $\overline{B}^{C_K}$. We will show that the point where we attach the branch $\mathfrak{T}_{\overline{e}_i}$ is $\overline{x}_i={\rm F}_{\mathfrak{T}^{(K)}}(d_i)$. 
       
Consider the tree $\bar{\mathfrak{T}}$ obtained from $\mathfrak{T}^{(K)}$ by attaching the trees $(\mathfrak{T}_{\overline{e}_i})_{i\in\N}$ to the points $(\overline{x}_i)_{i\in\N}$. The following lemma states that $\bar{\mathfrak{T}}$ has the same distribution as the CRT.
More precisely we will show that the pair $(\mathfrak{T},\mathfrak{T}^{(K)})$ has the same distribution as $(\overline{\mathfrak{T}},\mathfrak{T}^{(K)})$. In particular, the distribution of $\overline{\mathfrak{T}}$ is the same as that of $\mathfrak{T}$ even when conditioned on $\mathfrak{T}^{(K)}$.
\begin{lemma}\label{lem:treerepresentation}
For each $K\in\N$, the pair $(\bar{\mathfrak{T}}, \mathfrak{T}^{(K)})$ has the same distribution as $(\frak{T},\mathfrak{T}^{(K)})$.
\end{lemma}

\begin{proof}

We will rely on a particular discrete model that converges to the CRT.
Let us consider a critical Poisson Galton-Watson tree $\mathcal{T}_n$, conditioned to have size $n$. This is known to be the uniform tree on $n$ vertices (see Section 2 in \cite{AldousCRT1}). 

 Let $U_1,\ldots, U_K$ be uniform points in $[0,1]$. Let $\tilde{h}_n$ be the depth-first search around $\mathcal{T}_n$ and $\gamma$ defined as in \eqref{eq:defofgamma}. Then we have that $(V_i)_{i=1,\dots,K}:=(\tilde{h}_n(\gamma(U_i)))_{i=1,\dots,K}$ are i.i.d.~and uniform in the vertices of $\mathcal{T}_n$.

We can construct the $K$-skeleton associated with $\mathcal{T}_n$ and $V_1,\ldots, V_K$ that we denote by $\mathcal{T}_n^K$. This skeleton is composed of a shape $\tilde{\mathcal{T}}_n^K$ (which, for high $n$, will have $2K-1$ edges) and lengths associated to those edges of $\tilde{\mathcal{T}}_n^K$ which we denote $l_1^n,\ldots, l^n_{2K-1}$.
Let $({\bf e}_t)_{0\leq t \leq 1}$ be a normalized Brownian excursion independent of $(U_i)_{i=1,\dots,K}$ and 
$\mathfrak{T}^{(K)}$ be the reduced subtree $\mathcal{T}_{2{\bf e}}(V_1,\dots,V_K)$. It is not hard to see that $\mathcal{T}^{(K)}$ has $2K-1$ edges whose lengths we denote as $l_1,\dots,l_{2K-1}$.
 By the convergence results in \cite{AldousCRT3} (more specifically, Theorem 23 together with Theorem 15), we know that  $\tilde{\mathcal{T}}_n^K$ converges in distribution to the shape of $\mathcal{T}^{(K)}$ and $\frac{l_1^n}{\sqrt n},\ldots, \frac{l^n_{2K-1}}{\sqrt n}$ converges in distribution to $l_1,\ldots, l_{2K-1}$. 
 
 By virtue of the Skorohod representation theorem we can (and will) assume that the convergence stated above holds almost surely instead of in distribution.

 Let $v_1,\ldots, v_{L_n}$ the vertices of $\mathcal{T}^{K}_n$, where $L_n=\sum_{j=1}^{2K-1} l_j^{n}$. Since we are working with a Poisson Galton-Watson tree conditioned on $n$ vertices (which, as we have said, corresponds to a uniform random tree on $n$ vertices), we can see that conditionally on $\tilde{\mathcal{T}}_n^K$ and the lengths of the edges  $l_1^n,\ldots, l^n_{2K-1}$, the entire corresponding Galton-Watson tree is obtained by adding  to each vertex $v_i$ of  the skeleton independent Poisson-Galton-Watson trees $\mathcal{T}_n^{v_i}$ conditioned to have total size $n$. More precisely, conditionally on $\tilde{\mathcal{T}}_n^K$ and $l_1^n,\ldots, l^n_{2K-1}$, the tuple of branches $(\mathcal{T}_n^{v_i})_{v_i\in V(\tilde{\mathcal{T}}_n^{v_i})}$ is distributed uniformly on the (finite) set of tuples of trees that have total cardinality equal to $n$.
 
 Therefore, the process
 \[H_n(i):=\sum_{j=1}^{i}|\mathcal{T}_n^{v_j} |\]
 is a sum of independent random variables (each of which is distributed as the cardinality of a critical  Poisson-Galton-Watson tree) conditioned on having total sum equal to $n$. Moreover, by Lemma \ref{lem:assumptionhtforiic}, those random variables are in the domain of attraction of an $1/2$-stable law. Therefore, from Theorem 4 of \cite{liggett1968invariance}, we get that 
 \begin{equation}\label{eq:jumpcvg}
 (n^{-1}H_n(\lfloor n^{1/2} t\rfloor))_{t\in[0,n^{-1/2}L_n]}\stackrel{d}{\to} (S^{\frac{1}{2}}(t))_{t\in[0,C_K]}
 \end{equation}
as $n\to\infty$, where $S^{\frac{1}{2}}$ is a $\frac{1}{2}$-stable subordinator conditioned on $S^{\frac{1}{2}}(C_K)=1$.

 Since the local time at the origin of a Brownian motion is the inverse of a $\frac{1}{2}$-stable subordinator, we get that 
 \begin{equation}\label{localsub}
( (\overline{L}^{C_K})^{-1})_{t\in[0,1]}\stackrel{d}{=}((S^{\frac{1}{2}})^{-1}(t))_{t\in[0,1]}
 \end{equation} as $n\to\infty$, where $(S^{\frac{1}{2}})^{-1}(t):=\inf\{s:S^{\frac{1}{2}}(s)> t\}$. 
  
 On the other hand, since the law of $\T_n$ is uniform on the trees with $n$ vertices, given $\T_n^K$ and the sizes of the branches $|\T_n^{v_i}|_{i=1,\dots,L_n}$, the trees $(\T_n^{v_i})_{i=1,\dots,L_n}$ are independent uniform trees conditioned on the sizes $|\T_n^{v_i}|_{i=1,\dots,L_n}$. Therefore, recalling that a uniform tree on $m$ vertices scales to the CRT as $m\to\infty$, the sequence of branches converges to a sequence of scaled CRTs, where the scaling factors are given by the sizes of the jumps of $S^{\frac{1}{2}}$. More precisely, let $(d_i)_{i\in\mathbb{N}}$ be an enumeration of the discontinuities of $S_{1/2}$. Let $s_i=S_{1/2}(d_i)-S_{1/2}(d_i-)$ be the size of the $i$-th jump. Then, it is not hard to see that 
 \begin{equation}\label{eq:cvgtrees}(\T_n^{v_i})_{i=1,\dots,L_n}\stackrel{d}{\to} (\T_i)_{i\in\N}.
 \end{equation} as $n\to\infty$, where $(\T_i)_{i\in\N}$ is a sequence of independent CRTs with distances scaled by $(\sqrt{s_i})_{i\in\N}$ and volumes scaled by $(s_i)_{i\in\N}$. 
  
Furthermore, since the location of the jumps of $H_n$ records the location of the branches that hang off $\T_n^K$, displays \eqref{eq:jumpcvg} and \eqref{eq:cvgtrees} imply that $\T_n$ converges to a tree constructed by $\mathfrak{T}^{(K)}$ by attaching independent CRTs scaled by factors $(s_i)_{i\in\N}$ and located at positions $(F_K(d_i))_{i\in\N}$. 

Finally, Lemma 4.10 in \cite{pitman2006combinatorial} states that conditionally given the interval partition generated by its zero set, the
excursion of the Brownian bridge $\overline{B}^{C_K}$ over each interval of length $t$ is distributed as a Brownian excursion of length $t$, independently for the different intervals.
Therefore, taking into account \eqref{localsub}  and \eqref{eq:cvgtrees} by  we get that the description above coincides with that of the claim of the Lemma (in terms of the conditioned Brownian bridge).
\end{proof}

\subsection{An SSBM on a finite tree used to approximate the Brownian motion on the CRT}\label{eq:aftermuerto}

In this section we explain how to build an SSBM whose law is that of a Brownian motion projected on the $K$-reduced tree of a CRT. 
We need, three elements: a random tree, a law on subordinators and the law of a measure on our random tree. The tree will be given by $\mathfrak{T}^{(K)}$ obtained from the line-breaking construction.

 For any fixed realization of the CRT, the inverse local time at the root of the Brownian motion on a CRT is a subordinator. Therefore, under the randomness of the CRT, the inverse local time becomes a random subordinator. Let us denote $\overline{\mathbb{F}}$ the law of its (random) Laplace exponent.
 
  Finally, by the construction in Section~\ref{sect_this_sucks}, we know that we can define a $1/2$-stable Poisson point process on $\mathfrak{T}^{(K)}$ conditioned on having total volume 1  by setting $\mu=\sum_{i} \bigl((\overline{L}^{C_K})^{-1}_{d_i}-(\overline{L}^{C_K})^{-1}_{d_i^-}\bigr)\delta_{\overline{x}_i}$. Let us denote $\overline{\mathbb{M}}^{(1/2)}$ its law.

Recall the notations of Section~\ref{Sect_finite_SSBM}.
Let $B^{K\text{-ssbm}}$ be the $\mathfrak{T}^{(K)}$-SSBM corresponding to the laws $\overline{\mathbb{F}}$ and $\overline{\mathbb{M}}^{(1/2)}$ defined above.

  We conjecture that, with the same ideas used for the proof of Theorem \ref{prop:alternativeexrepssionforziic}, it is possible to deduce from Lemma \ref{lem:treerepresentation} that
  \begin{equation}\label{prop_finite_ssbm}
(B^{K\text{-crt}}_t)_{t\geq0}\stackrel{d}{=}(B^{K\text{-ssbm}}_t)_{t\geq0} 
\end{equation}
 for any $K\in\N$.

\begin{remark}\label{rmk:last}
By Proposition 2.2 in \cite{rwrt}, the projection $\pi_{\frak{T}^{(K)}}:\frak{T}\to\frak{T}^{(K)}$ converges uniformly to the identity as $K\to\infty$. In particular, $B^{K\text{-crt}}$ can be made arbitrarily close to $B^{\text{CRT}}$ by choosing $K$ large enough. On the other hand, by \eqref{prop_finite_ssbm}, for each $K\in\N$, $B^{K\text{-crt}}$ is an SSBM on a finite tree. Therefore, $B^{\text{CRT}}$ can be seen as a limit of SSBMs. Those SSBMs characterize $B^{\text{CRT}}$ in the following sense: If there is a process $(W_t)_{t\geq 0}$ taking values in the CRT such that for any $K\in\N$, we have that $(\pi_{\frak{T}^{(K)}}(W_t))_{t\geq0}$ is distributed as $(B^{K\text{-crt}}_t)_{t\geq0}$, then $W$ has the same law as $B^{\text{CRT}}$. This follows from the aforementioned convergence of $\pi_{\frak{T}^{(K)}}$ towards the identity.
\end{remark}

\subsection{Convergence of SSBMs on the CRT}\label{sect_final}
In this section, we discuss informally the relations between the convergence of discrete models towards $B^{\text{CRT}}$ and the convergence of their respective projections to the $K$-reduced sub-trees.
Our interest stems from the fact that, as we have said in \eqref{prop_finite_ssbm}, we conjecture that the projection of $B^{\text{CRT}}$ is an SSBM. We will discuss these relations using the model of Random walks on critical Galton-Watson trees and Random walks on the range of critical branching random walks (in high dimensions).

Let ${\mathcal T}_n$ be a critical Galton Watson tree conditioned on $|{\mathcal T}_n|=n$ and $(X^n_l)_{l\in\N}$ be a simple random walk on ${\mathcal T}_n$. In \cite{rwrt} Croydon showed that $X^n$ converges to $B^{\text{CRT}}$. It can be shown that this convergence implies the convergence of the respective projections onto the reduced sub-trees. Conversely, as explained in Remark \ref{rmk:last}, the convergence of the projections onto the reduced sub-trees implies that $X^n$ scales to $B^{\text{CRT}}$.  

We now pass to describe the model of critical branching random walks in $\mathbb{Z}^d$, $d\in\N$.
Let $\T_n$ be a critical Galton Watson tree conditioned on $|\T_n|=n$. Let $E(\T_n)$ be the set of edges of $\T_n$ and $(L_e)_{e\in E(\T_n)}$ be an i.i.d. sequence distributed uniformly in the $2d$ unitary vectors of $\mathbb{Z}^d$. For any $v\in\T_n$, let $[\text{root},v]$ denote the set of edges in the path from the root to $v$. We define
\[\Phi_n(v):=\sum_{e\in[\text{root},v]} L_e.\]
Here $v$ represents the genealogical label of a particle and $\Phi_n(v)$ its position. The model is that of particles performing branching and jumping with symmetric transition probabilities.

Let us now describe the range of the critical branching random walks. Let $G_n$ be the subgraph of $\mathbb{Z}^d$ induced by the mapping $\Phi_n:\T_n:\to\Z^d$.

Next, consider the $(X^n_k)_{k\geq 0}$ the simple random walk on $G_n$ started at $\boldsymbol{o}$. It was proved in~\cite{simplelabyrinth} that $X^n$ converges for $d$ large, after appropriate rescaling, to an object called the Brownian motion on the ISE. This object can be obtained from the Brownian motion on the CRT by an appropriate isometric embedding into $\Z^d$. By considering the $K$-CRT and embedding it using the same embedding we obtain an object that we call the $K$-ISE.

If we consider the Brownian motion on the ISE projected onto the $K$-ISE then the resulting object is an SSBM (on the CRT) where each segment of the finite tree is embedded using Brownian motions in $\Z^d$. Furthermore, we know that this object appears as the scaling limit of certain finite reduced critical models (this statement is implicit in~\cite{highdimensionallabyrinth}) and converges as $K$ goes to infinity to the Brownian motion on the ISE (also a consequence of~\cite{highdimensionallabyrinth}).

 \vspace{0.5cm}
 
{\bf Acknowledgements} We would like to thank Louigi Addario-Berry for his very valuable input simplifying the proof of Lemma~\ref{lem:treerepresentation} and David Croydon for his useful answers to questions concerning fine properties of the Brownian motion on the CRT and its local time.

 \bibliographystyle{alpha}

\begin{thebibliography}{AGdHS08}

\bibitem[AGdHS08]{ipc}
O.~Angel, J.~Goodman, F.~den Hollander, and G.~Slade.
\newblock Invasion percolation on regular trees.
\newblock {\em Ann. Probab.}, 36(2):420--466, 2008.

\bibitem[AGM13]{AngelGoodmanMerle2013}
O.~Angel, J.~Goodman, and M.~Merle.
\newblock Scaling limit of the invasion percolation cluster on a regular tree.
\newblock {\em Ann. Probab.}, 41(1):229--261, 2013.

\bibitem[Ald91a]{AldousCRT1}
D.~Aldous.
\newblock The continuum random tree. {I}.
\newblock {\em Ann. Probab.}, 19(1):1--28, 1991.

\bibitem[Ald91b]{AldousCRT2}
D.~Aldous.
\newblock The continuum random tree. {II}. {A}n overview.
\newblock In {\em Stochastic analysis ({D}urham, 1990)}, volume 167 of {\em
  London Math. Soc. Lecture Note Ser.}, pages 23--70. Cambridge Univ. Press,
  Cambridge, 1991.

\bibitem[Ald93]{AldousCRT3}
D.~Aldous.
\newblock The continuum random tree. {III}.
\newblock {\em Ann. Probab.}, 21(1):248--289, 1993.

\bibitem[App09]{levy}
D.~Applebaum.
\newblock {\em L{\'e}vy processes and stochastic calculus}, volume 116 of {\em
  Cambridge Studies in Advanced Mathematics}.
\newblock Cambridge University Press, Cambridge, second edition, 2009.

\bibitem[ALW17]{athreya2017invariance}S. Athreya, W. L{\"o}hr and A. Winter.
\newblock Invariance principle for variable speed random walks on trees.
\newblock{\em
  The Annals of Probability}, 45(2):625--667, 2017.
   
\bibitem[BC{\v{C}}R15]{rtrw}
G.~{Ben Arous}, M.~Cabezas, J.~{\v{C}}ern{\`y}, and R.~Royfman.
\newblock Randomly trapped random walks.
\newblock {\em The Annals of Probability}, 43(5):2405--2457, 2015.

\bibitem[BCF16a]{simplelabyrinth}
G{\'e}rard {Ben Arous}, Manuel Cabezas, and Alexander Fribergh.
\newblock Scaling limit for the ant in a simple high-dimensional labyrinth.
\newblock {\em Probability Theory and Related Fields}, 174 (1-2):553--646, 2019.
  
\bibitem[BCF16b]{highdimensionallabyrinth}
G{\'e}rard {Ben Arous}, Manuel Cabezas, and Alexander Fribergh.
\newblock Scaling Limit for the Ant in High-Dimensional Labyrinths.
\newblock {\em Communications on Pure and Applied Mathematics}, 72 (4): 669--763, 2019.



\bibitem[Bil68]{Billingsley68}
P.~Billingsley.
\newblock {\em Convergence of probability measures}.
\newblock John Wiley \& Sons Inc., New York, 1968.

\bibitem[BK06]{BarlowKumagai2006}
M.~Barlow and T.~Kumagai.
\newblock Random walk on the incipient infinite cluster on trees.
\newblock {\em Illinois J. Math.}, 50(1-4):33--65 (electronic), 2006.

\bibitem[CJ01]{chassaing2001vervaat}
P.~Chassaing and S.~Janson.
\newblock A vervaat-like path transformation for the reflected brownian bridge
  conditioned on its local time at 0.
\newblock {\em Annals of probability}, pages 1755--1779, 2001.

\bibitem[Cro08]{rwrt}
D.~Croydon.
\newblock Convergence of simple random walks on random discrete trees to
  {B}rownian motion on the continuum random tree.
\newblock {\em Ann. Inst. Henri Poincar{\'e} Probab. Stat.}, 44(6):987--1019,
  2008.
  
 \bibitem[Cro18]{scaling}
 D.~Croydon
  \newblock Scaling limits of stochastic processes associated with resistance forms.
\newblock {\em Ann. Inst. Henri Poincar{\'e} Probab. Stat.},   54(4):1939--1968,
  2018.
  
\bibitem[Cro12]{Croydon2012}
D.~Croydon.
\newblock Scaling limit for the random walk on the largest connected component
  of the critical random graph.
\newblock {\em Publ. Res. Inst. Math. Sci.}, 48(2):279--338, 2012.

\bibitem[Fel71]{fel71}
W.~Feller.
\newblock {\em An introduction to probability theory and its applications.
  {V}ol. {II}.}
\newblock Second edition. John Wiley \& Sons Inc., New York, 1971.

\bibitem[FOT10]{fukushima2010dirichlet}
Fukushima, M., Oshima, Y. \& Takeda, M.
  \newblock Dirichlet forms and symmetric Markov processes,
  \newblock de Gruyter, 2010.


\bibitem[Kes86]{Kesten}
H.~Kesten.
\newblock Subdiffusive behavior of random walk on a random cluster.
\newblock {\em Ann. Inst. H. Poincar{\'e} Probab. Statist.}, 22(4):425--487,
  1986.

\bibitem[Kig95]{kigami1995harmonic}
J.~Kigami.
\newblock Harmonic calculus on limits of networks and its application to
  dendrites.
\newblock {\em Journal of Functional Analysis}, 128(1):48--86, 1995.

\bibitem[Kre95]{krebs1995brownian}
W.~Krebs.
\newblock Brownian motion on the continuum tree.
\newblock {\em Probability theory and related fields}, 101(3):421--433, 1995.

\bibitem[Lig68]{liggett1968invariance}
Thomas~M Liggett.
\newblock An invariance principle for conditioned sums of independent random
  variables.
\newblock {\em Journal of Mathematics and Mechanics}, 18(6):559, 1968.


\bibitem[MR06]{marcus2006markov}
M. Marcus  and J. Rosen.
	\newblock Markov processes, Gaussian processes, and local times,
	\newblock{Cambridge University Press, 2006}

\bibitem[Pit99]{pitman1999sde}
J.~Pitman.
\newblock The sde solved by local times of a brownian excursion or bridge
  derived from the height profile of a random tree or forest.
\newblock {\em The Annals of Probability}, 27(1):261--283, 1999.

\bibitem[Pit06]{pitman2006combinatorial}
J. Pitman.
\newblock {\em Combinatorial Stochastic Processes: Ecole D'Et{\'e} de
  Probabilit{\'e}s de Saint-Flour XXXII-2002}.
\newblock Springer, 2006.

\bibitem[Whi02]{whi02}
W.~Whitt.
\newblock {\em Stochastic-process limits}.
\newblock Springer Series in Operations Research. Springer-Verlag, New York,
  2002.
\newblock An introduction to stochastic-process limits and their application to
  queues.

\bibitem[WW83]{WilkinsonWillemsen1983}
D.~Wilkinson and J.~F. Willemsen.
\newblock Invasion percolation: a new form of percolation theory.
\newblock {\em J. Phys. A}, 16(14):3365--3376, 1983.

\end{thebibliography}
\def\cprime{$'$}

\end{document}